\documentclass[11pt,reqno]{amsart}

\usepackage[utf8]{inputenc}
\usepackage[T1]{fontenc}
\usepackage{amsmath,amssymb,amsthm}
\usepackage[margin=3cm]{geometry}
\usepackage{enumitem}
\usepackage{microtype}
\usepackage{booktabs}
\usepackage{array}
\usepackage{tikz}
\usepackage{listings}
\usepackage{url}

\theoremstyle{plain}
\newtheorem{theorem}{Theorem}[section]
\newtheorem{proposition}[theorem]{Proposition}
\newtheorem{lemma}[theorem]{Lemma}
\newtheorem{corollary}[theorem]{Corollary}
\newtheorem{conjecture}[theorem]{Conjecture}
\theoremstyle{definition}
\newtheorem{definition}[theorem]{Definition}

\theoremstyle{remark}
\newtheorem{remark}[theorem]{Remark}

\newcommand{\Z}{\mathbb Z}
\newcommand{\F}{\mathbb F}
\newcommand{\Lamone}{\Lambda_{\mathbf1}}
\newcommand{\Sred}{\mathsf S_{\mathrm{red}}}
\newcommand{\BG}{B_{\Gamma}}
\newcommand{\rad}{\operatorname{rad}}
\DeclareMathOperator{\ex}{exp}

\begin{document}
\raggedbottom

\title{Algebraic characterizations of generating and affinely generating $\Gamma$-magic maps and $\Gamma$-distance magic labelings on regular graphs}

\author{Ahmet Batal}
\address{Department of Mathematics, Izmir Institute of Technology, 35430, Urla, Izmir, Turkey}
\email{ahmetbatal@iyte.edu.tr}

\subjclass[2020]{05C78, 05C25, 05C50, 20K01}
\keywords{Group-valued magic map, distance magic labeling, finite abelian group, Smith normal form, reduced adjacency Smith group, Cayley graph, strongly regular graph}

\begin{abstract}
A graph $G$ of order $n$ is \emph{distance magic} if it admits a bijection from $V(G)$ to $\{1,\ldots,n\}$ whose open-neighborhood sums are constant. Given an abelian group $\Gamma$ of order $n$, the graph $G$ is \emph{$\Gamma$-distance magic} if it admits a bijection $f\colon V(G)\to\Gamma$ whose open-neighborhood sums are constant in $\Gamma$. The graph $G$ is \emph{group distance magic} if it is $\Gamma$-distance magic for every abelian group $\Gamma$ of order $n$.

More generally, for an arbitrary finite abelian group $\Gamma$, a $\Gamma$-magic map is a map $f\colon V(G)\to\Gamma$ whose open-neighborhood sums are constant. We introduce two generation conditions for such maps. We call $f$ \emph{generating} if its labels generate $\Gamma$, and \emph{affinely generating} if its pairwise differences generate $\Gamma$.

Let $G$ be regular, and write the invariant factor decomposition of $\Gamma$ as
\[
  \Gamma\cong \Z/d_1\Z\oplus\cdots\oplus\Z/d_r\Z,
  \qquad d_1\mid\cdots\mid d_r,
\]
and put $\BG=\bigoplus_{i=1}^{r-1}\Z/d_i\Z$, with $\BG=\{0\}$ when $r=1$. Let $\overline A_m$ denote the endomorphism induced by the adjacency operator on the quotient of $(\Z/m\Z)^{V(G)}$ by the constant vectors. We prove that $G$ admits a generating $\Gamma$-magic map if and only if $\BG\hookrightarrow\ker\overline A_{d_r}$, whereas it admits an affinely generating $\Gamma$-magic map if and only if $\Gamma\hookrightarrow\ker\overline A_{d_r}$. When $|V(G)|=|\Gamma|$, the graph is $\Gamma$-distance magic if and only if the latter embedding can be chosen with vertex-separating image. If the reduced adjacency operator is nonsingular over $\mathbb Q$, the first two conditions become subgroup conditions in the reduced adjacency Smith group.

Cichacz and Froncek conjectured that every distance magic graph is group distance magic. Using the affine criterion, we construct a $6$-regular distance magic graph of order $27$ that admits no affinely generating $(\Z/3\Z)^3$-magic map. Thus their conjecture fails even after bijectivity is weakened to affine generation. We propose the generating group distance magic conjecture that every distance magic graph of order $n$ admits a generating $\Gamma$-magic map for every abelian group $\Gamma$ of order $n$. We prove this conjecture for regular distance magic graphs when $\Gamma$ is generated by at most two elements, and consequently for regular distance magic graphs of cube-free order. For Cayley graphs on elementary abelian $p$-groups, our nonexistence result when the regularity is not divisible by $p$, combined with a known sufficient condition when the regularity is divisible by $p$, yields an exact criterion for the existence of a $\Gamma$-distance magic labeling over some abelian group $\Gamma$ of the graph order. Further applications concern the Hamming relation graphs, strongly regular graphs, and incidence graphs of symmetric designs.
\end{abstract}

\maketitle
\pagestyle{plain}

%======================================================================
\section{Introduction}

Let $G$ be a finite simple graph and let $\Gamma$ be a finite abelian group written additively. For $v\in V(G)$, let $N_G(v)$ denote the neighborhood of $v$ in $G$. A map
\[
  f\colon V(G)\longrightarrow\Gamma
\]
is a \emph{$\Gamma$-magic map} if there exists $\gamma\in\Gamma$ such that
\[
  \sum_{u\in N_G(v)}f(u)=\gamma
  \qquad\text{for every }v\in V(G).
\]
A $\Gamma$-distance magic labeling is a bijective $\Gamma$-magic map with $|V(G)|=|\Gamma|$~\cite{Fro13,CFSZ16}. An \emph{ordinary distance magic labeling} of a graph of order $N$ is a bijection $\ell\colon V(G)\to\{1,\ldots,N\}$ whose open-neighborhood sums are constant. A graph admitting such a labeling is called \emph{distance magic}. A graph of order $N$ is \emph{group distance magic} if it is $\Gamma$-distance magic for every abelian group $\Gamma$ of order $N$.

For a family $(x_i)_{i\in I}$ of elements of $\Gamma$, write $\langle x_i\mid i\in I\rangle$ for the subgroup of $\Gamma$ generated by this family.

To separate the algebraic content of bijectivity from the requirement that all labels be distinct, we consider two generation conditions on $\Gamma$-magic maps. Both are implied by bijectivity when $|V(G)|=|\Gamma|$.

\begin{definition}\label{def:generation-conditions}
Let $f\colon V(G)\to\Gamma$ be a $\Gamma$-magic map. We say that $f$ is \emph{generating} if
\[
  \left\langle f(v)\mid v\in V(G)\right\rangle=\Gamma.
\]
The \emph{affine difference group} of $f$ is
\[
  D(f)=
  \left\langle f(v)-f(w)\mid v,w\in V(G)\right\rangle.
\]
We say that $f$ is \emph{affinely generating} if $D(f)=\Gamma$.
\end{definition}

Thus
\[
  \Gamma\text{-distance magic}
  \Longrightarrow
  \text{affinely generating }\Gamma\text{-magic}
  \Longrightarrow
  \text{generating }\Gamma\text{-magic}.
\]
These implications are recorded in Proposition~\ref{prop:affine-basic}.

Let $G$ be $k$-regular with adjacency matrix $A$, and write $\mathbf1$ for the all-one vector indexed by $V(G)$. For $m\ge1$, let $A_m$ be the endomorphism of $(\Z/m\Z)^{V(G)}$ obtained by reducing $A$ modulo $m$. Since $A_m\mathbf1=k\mathbf1$, the submodule $(\Z/m\Z)\mathbf1$ of constant vectors is $A_m$-invariant. Hence $A_m$ induces an endomorphism
\[
  \overline A_m\colon
  \frac{(\Z/m\Z)^{V(G)}}{(\Z/m\Z)\mathbf1}
  \longrightarrow
  \frac{(\Z/m\Z)^{V(G)}}{(\Z/m\Z)\mathbf1},
  \qquad
  \overline A_m([x])=[A_mx].
\]
We call $\overline A_m$ the \emph{reduced adjacency operator modulo $m$}. For a homomorphism $T$, we write $\ker T$ for its kernel. Let $x_v$ denote the coordinate of $x$ indexed by $v$. For a class $[x]$, the difference $x_v-x_w$ is independent of the chosen representative. A subgroup $W\le\ker\overline A_m$ is \emph{vertex-separating} if, for every pair of distinct vertices $v,w$, some $[x]\in W$ satisfies $x_v-x_w\ne0$ in $\Z/m\Z$.
For a finite abelian group $H$, its exponent is
$
  \ex(H)=\min\{m\ge1:mh=0\text{ for every }h\in H\}.
$
The notation $H\hookrightarrow K$ means that there exists an injective homomorphism from $H$ to $K$, or equivalently that $H$ is isomorphic to a subgroup of $K$. 

The main result gives a common modular description of generating $\Gamma$-magic maps, affinely generating $\Gamma$-magic maps, and $\Gamma$-distance magic labelings.

\begin{theorem}\label{thm:main-characterization}
Let $G$ be a regular graph and let $\Gamma$ be a finite abelian group with invariant factor decomposition
\[
  \Gamma\cong
  \Z/d_1\Z\oplus\cdots\oplus\Z/d_r\Z,
  \qquad
  d_1\mid\cdots\mid d_r.
\]
Put
\[
  \BG=\bigoplus_{i=1}^{r-1}\Z/d_i\Z,
\]
where $\BG=\{0\}$ when $r=1$. Then $d_r=\ex(\Gamma)$, and the following statements hold.
\begin{enumerate}[label=\emph{(\alph*)},leftmargin=2.2em]
\item The graph $G$ admits a generating $\Gamma$-magic map if and only if
\[
  \BG\hookrightarrow\ker\overline A_{d_r}.
\]
\item The graph $G$ admits an affinely generating $\Gamma$-magic map if and only if
\[
  \Gamma\hookrightarrow\ker\overline A_{d_r}.
\]
\item If $|V(G)|=|\Gamma|$, then $G$ admits a $\Gamma$-distance magic labeling if and only if
\[
  \Gamma\hookrightarrow\ker\overline A_{d_r}
\]
with vertex-separating image.
\end{enumerate}
\end{theorem}

For the Smith-group formulation, use integral coefficients and put
\[
  \Lamone=\Z^{V(G)}/\Z\mathbf1.
\]
Since $\Z\mathbf1$ is $A$-invariant, $A$ induces
\[
  \overline A\colon\Lamone\longrightarrow\Lamone,
  \qquad
  \overline A([x])=[Ax].
\]
When $\overline A$ is nonsingular over $\mathbb Q$, its cokernel is finite. We call
\[
  \Sred(G)=\operatorname{coker}\overline A
  =\Lamone/\overline A(\Lamone)
\]
the \emph{reduced adjacency Smith group} of $G$. The term \emph{Smith group} reflects that the isomorphism type of this cokernel is determined by the Smith normal form of $\overline A$. Reduction of this integral operator modulo $m$ is identified in Section~\ref{sec:smith} with the modular operator $\overline A_m$ above. The reduced-Smith form of Theorem~\ref{thm:main-characterization} gives the following main corollary.

\begin{corollary}\label{cor:main-smith}
Let $G$ be a $k$-regular graph with $k>0$, let $\Gamma$ and $\BG$ be as in Theorem~\ref{thm:main-characterization}, and assume that $\overline A$ is nonsingular over $\mathbb Q$. Then the following statements hold.
\begin{enumerate}[label=\emph{(\alph*)},leftmargin=2.2em]
\item The graph $G$ admits a generating $\Gamma$-magic map if and only if
\[
  \BG\hookrightarrow\Sred(G).
\]
If $r\ge2$, then necessarily
\[
  d_{r-1}\mid\ex(\Sred(G))
  \qquad\text{and}\qquad
  \frac{|\Gamma|}{d_r}\mid\frac{|\det A|}{k}.
\]
\item The graph $G$ admits an affinely generating $\Gamma$-magic map if and only if
\[
  \Gamma\hookrightarrow\Sred(G).
\]
Consequently,
\[
  d_r\mid\ex(\Sred(G))
  \qquad\text{and}\qquad
  |\Gamma|\mid\frac{|\det A|}{k}.
\]
\item In particular, if $|V(G)|=|\Gamma|$ and $G$ admits a $\Gamma$-distance magic labeling, then the necessary conditions in part~\emph{(b)} hold.
\end{enumerate}
\end{corollary}

The identity
\[
  |\Sred(G)|=\frac{|\det A|}{k}
\]
is proved in Proposition~\ref{prop:smith-exact}; this accounts for the determinant divisibilities in the corollary.

Thus, in the reduced-Smith criterion, ordinary generation discards one largest invariant factor of the labeling group, whereas affine generation retains the full group. A $\Gamma$-distance magic labeling is subject to the same subgroup condition as affine generation, together with the vertex-separation requirement in the modular kernel.

The contrast between the reduced-Smith criteria for generating and affinely generating magic maps is central to our analysis of the conjecture of Cichacz and Froncek that every distance magic graph is group distance magic, originally formulated in~\cite[Conjecture~29]{CF16} and later restated in~\cite[Conjecture~1.7]{CDF20}. We show that the conjecture fails even after bijectivity is weakened to affine generation. More precisely, writing $\F_q$ for the field with $q$ elements when $q$ is a prime power, we construct a $6$-regular distance magic graph $G$ of order $27$ for which
\[
  \dim_{\F_3}\ker\overline A_3=2.
\]
Proposition~\ref{prop:group-distance-kernel-test}\emph{(a)} then shows that $G$ admits a generating $(\Z/3\Z)^3$-magic map but no affinely generating one. We therefore propose the generating group distance magic conjecture, in which affine generation is replaced by ordinary generation. The conjecture is stated without a regularity assumption, as in the original group distance magic conjecture. For regular distance magic graphs we prove it whenever the labeling group is generated by at most two elements. In particular, it holds for regular distance magic graphs of cube-free order.

Combining our nonexistence result with the known sufficient direction gives an exact distance-magic existence criterion for Cayley graphs on elementary abelian $p$-groups. If $P\cong(\Z/p\Z)^n$, then
\[
\begin{aligned}
&\operatorname{Cay}(P,S)\text{ admits a $\Gamma$-distance magic labeling}\\
&\text{for some abelian group $\Gamma$ of order $|P|$}
\end{aligned}
\quad\Longleftrightarrow\quad
p\mid|S|.
\]
This criterion yields exact existence results for the Hamming relation graphs $H_j(d,q)$ and their special cases. In each feasible case, the elementary abelian group underlying the Cayley graph can be used as the labeling group. These results do not classify an arbitrary prescribed labeling group. For comparison, Anholcer et al.\ proved that $Q_d$ is $\Gamma$-distance magic for every abelian group $\Gamma$ of order $2^d$ if and only if $d$ is even~\cite{ACFSQ21}.

A closely related nonbijective framework is that of group vertex magic labelings introduced by Kamatchi et al.~\cite{Kam20}. In that setting a labeling takes values in $\Gamma\setminus\{0\}$, whereas the $\Gamma$-magic maps considered here allow the zero element. Our emphasis is instead on generation by the labels and by their pairwise differences.

The vertex-separation viewpoint in Theorem~\ref{thm:main-characterization}\emph{(c)} has a binary predecessor in the theory of XOR-magic graphs. For a graph of order $2^n$, an open XOR-magic labeling is precisely a $(\Z/2\Z)^n$-distance magic labeling with magic constant $0$. A closed XOR-magic labeling is the corresponding closed-neighborhood analogue~\cite{Bat26}. For closed XOR-magic graphs, Siehler~\cite[Theorem~2]{Sie19} obtained nullity and row-distinctness tests from the nullspace of the closed-neighborhood matrix. Batal~\cite[Proposition~8.1]{Bat26} gave a matrix characterization for closed and open XOR-magic labelings with zero magic constant in which suitable null vectors are assembled as columns and the resulting rows are pairwise distinct. In the open case the relevant neighborhood operator is the adjacency matrix. Theorem~\ref{thm:main-characterization}\emph{(c)} extends this row-separation framework from elementary abelian $2$-groups to arbitrary finite abelian labeling groups and allows arbitrary magic constants through the reduced modular kernel.

Section~\ref{sec:modular} records some finite-abelian-group facts and proves the modular characterizations including Theorem~\ref{thm:main-characterization}. Section~\ref{sec:smith} connects the modular operators with the integral reduced adjacency operator, gives the reduced-Smith formulation, and gives an example showing that affine generation need not imply bijectivity. Section~\ref{sec:weakened-conjecture} records modular criteria for group distance magic and constructs a counterexample to the conjecture of Cichacz and Froncek, even with bijectivity weakened to affine generation. It then states the generating group distance magic conjecture and proves it for regular distance magic graphs when the labeling group has at most two generators, and consequently for regular distance magic graphs of cube-free order. Section~\ref{sec:applications} uses the modular and reduced-Smith criteria to study generating and affinely generating $\Gamma$-magic maps and $\Gamma$-distance magic labelings across several graph families. These include complete graphs, Cayley graphs on finite $p$-groups, Hamming relation graphs, classical forms graphs, strongly regular graphs, and incidence graphs of symmetric designs.

%======================================================================
\section{Modular characterizations}\label{sec:modular}

We first record elementary consequences of the affine difference group on regular graphs.

\begin{proposition}\label{prop:affine-basic}
Let $G$ be a $k$-regular graph and let $f\colon V(G)\to\Gamma$ be a $\Gamma$-magic map.
\begin{enumerate}[label=\emph{(\alph*)},leftmargin=2.2em]
\item For every $v_0\in V(G)$,
\[
  D(f)=\left\langle f(v)-f(v_0)\mid v\in V(G)\right\rangle.
\]
\item The image $f(V(G))$ is contained in a coset of $D(f)$, and $D(f)$ is the smallest subgroup of $\Gamma$ with this property.
\item For every $a\in\Gamma$, the translated map $f_a(v)=f(v)+a$ is magic and satisfies $D(f_a)=D(f)$.
\item If $f$ is affinely generating, then $f$ is generating.
\item If $f$ is a $\Gamma$-distance magic labeling, then $f$ is affinely generating.
\item If $f(v_0)=0$ for some $v_0\in V(G)$, then $f$ is affinely generating if and only if it is generating. Consequently, $G$ admits an affinely generating $\Gamma$-magic map if and only if it admits a generating $\Gamma$-magic map $g$ satisfying $g(v_0)=0$ for some $v_0\in V(G)$.
\end{enumerate}
\end{proposition}

\begin{proof}
Part~\emph{(a)} follows from
\[
  f(v)-f(w)=\bigl(f(v)-f(v_0)\bigr)-\bigl(f(w)-f(v_0)\bigr).
\]
For part~\emph{(b)}, part~\emph{(a)} gives $f(v)\in f(v_0)+D(f)$ for every $v\in V(G)$. If $f(V(G))\subseteq a+H$ for a subgroup $H\le\Gamma$, then every label difference belongs to $H$, so $D(f)\le H$. Thus $D(f)$ is the smallest subgroup whose coset contains the image of $f$.

For part~\emph{(c)}, if $\gamma$ is the magic constant of $f$, then regularity gives
\[
  \sum_{u\in N_G(v)}f_a(u)=\gamma+ka,
\]
and translation does not change label differences. Part~\emph{(d)} holds because every label difference lies in $\langle f(V(G))\rangle$. For part~\emph{(e)}, choose $v_0$ with $f(v_0)=0$. Since $f$ is a $\Gamma$-distance magic labeling, it is bijective and hence $f(V(G))=\Gamma$. Part~\emph{(a)} therefore yields $D(f)=\Gamma$.

For part~\emph{(f)}, if $f(v_0)=0$, then part~\emph{(a)} gives
\[
  D(f)=\langle f(V(G))\rangle,
\]
which proves the first assertion. If $h$ is affinely generating, fix $v_0\in V(G)$ and put $g(v)=h(v)-h(v_0)$. By part~\emph{(c)}, $g$ is magic and affinely generating, and $g(v_0)=0$, so the first assertion shows that $g$ is generating. The converse follows directly from the first assertion.
\end{proof}

We shall also use two standard facts on finite abelian groups. The first is the finite-abelian character duality and separation result in~\cite[Corollary~3.7 and Theorems~3.13--3.14]{ConradCharacters}, in its equivalent additive formulation. See also~\cite[Section~4.2, pp.~126--129]{Nathanson00}. The second fact is the invariant-factor interlacing for subgroups with cyclic quotient established in the proof of~\cite[Lemma~3.2]{Domokos17}; see also~\cite[Chapter~II, (4.3)]{Macdonald95} and~\cite[Chapter~I, (5.16)]{Macdonald95}.

\begin{lemma}[\cite{ConradCharacters}]\label{lem:duality}
Let $H$ be a finite abelian group, let $m$ be divisible by $\ex(H)$, and put
\[
  H^*=\operatorname{Hom}(H,\Z/m\Z).
\]
The evaluation map
\[
  \operatorname{ev}_H\colon H\longrightarrow\operatorname{Hom}(H^*,\Z/m\Z),
  \qquad
  \operatorname{ev}_H(h)(\chi)=\chi(h),
\]
is an isomorphism. In particular, $H^*\cong H$. If $K<H$ is a proper subgroup, then there exists a nonzero $\chi\in H^*$ such that $\chi|_K=0$.
\end{lemma}

\begin{lemma}[\cite{Domokos17}]\label{lem:cyclic-quotient}
Let
\[
  \Gamma\cong
  \Z/d_1\Z\oplus\cdots\oplus\Z/d_r\Z,
  \qquad
  d_1\mid\cdots\mid d_r,
\]
and let $H\le\Gamma$. If $\Gamma/H$ is cyclic, then
\[
  \bigoplus_{i=1}^{r-1}\Z/d_i\Z
  \hookrightarrow H.
\]
\end{lemma}

The next lemma describes how ordinary generation differs from affine generation.

\begin{lemma}\label{lem:generating-affine}
Let $G$ be regular and let $\Gamma$ be a finite abelian group. The graph $G$ admits a generating $\Gamma$-magic map if and only if there exists a subgroup $H\le\Gamma$ such that $\Gamma/H$ is cyclic and $G$ admits an affinely generating $H$-magic map.
\end{lemma}

\begin{proof}
Suppose first that $f\colon V(G)\to\Gamma$ is generating and magic, and put $H=D(f)$. Fix $v_0\in V(G)$ and define
\[
  g(v)=f(v)-f(v_0).
\]
Regularity implies that $g$ is magic. Its values lie in $H$, and $D(g)=H$, so $g$ is an affinely generating $H$-magic map. Since every label of $f$ lies in the coset $f(v_0)+H$ and the labels generate $\Gamma$, the quotient $\Gamma/H$ is generated by $f(v_0)+H$ and is therefore cyclic.

Conversely, suppose that $H\le\Gamma$, that $\Gamma/H$ is cyclic, and that $g\colon V(G)\to H$ is affinely generating and magic. Choose $a\in\Gamma$ such that $a+H$ generates $\Gamma/H$, and define
\[
  f(v)=g(v)+a.
\]
Regularity implies that $f$ is a $\Gamma$-magic map. The subgroup generated by its labels contains all pairwise differences, hence contains $D(f)=D(g)=H$. Its image in $\Gamma/H$ contains $a+H$, so it is all of $\Gamma/H$. Therefore the labels of $f$ generate $\Gamma$.
\end{proof}

The following lemma is the main algebraic bridge of the paper, translating the existence of affinely generating $H$-magic maps into the existence of injective homomorphisms from the character group $H^*$ into the modular kernel.

\begin{lemma}\label{lem:affine-param}
Let $G$ be a regular graph, let $H$ be a finite abelian group, and let $m\ge1$. Put
\[
  H^*=\operatorname{Hom}(H,\Z/m\Z).
\]
The following statements hold.
\begin{enumerate}[label=\emph{(\alph*)},leftmargin=2.2em]
\item If $f\colon V(G)\to H$ is an affinely generating $H$-magic map, then
\[
  \Psi_f\colon H^*\longrightarrow\ker\overline A_m,
  \qquad
  \Psi_f(\chi)=[\chi\circ f],
\]
is an injective homomorphism. In particular, the conclusion applies to every $H$-distance magic labeling.
\item Suppose that $m$ is divisible by $\ex(H)$, and fix $v_0\in V(G)$. Then $f\mapsto\Psi_f$ is a bijection between the following sets.
\begin{enumerate}[label=\emph{(\roman*)},leftmargin=2.3em]
\item The generating $H$-magic maps $f$ satisfying $f(v_0)=0$.
\item The injective homomorphisms
\[
  \Psi\colon H^*\hookrightarrow\ker\overline A_m.
\]
\end{enumerate}
\end{enumerate}
\end{lemma}

\begin{proof}
For part~\emph{(a)}, let $f$ be an affinely generating $H$-magic map with magic constant $\gamma$. For every $\chi\in H^*$,
\[
  A_m(\chi\circ f)=\chi(\gamma)\mathbf1,
\]
so $[\chi\circ f]\in\ker\overline A_m$. Thus $\Psi_f$ is a homomorphism. If $\Psi_f(\chi)=0$, then $\chi\circ f$ is constant, so $\chi$ vanishes on every label difference. Since $D(f)=H$, one has $\chi=0$. Hence $\Psi_f$ is injective.

For part~\emph{(b)}, let $f$ be a generating $H$-magic map satisfying $f(v_0)=0$. By Proposition~\ref{prop:affine-basic}\emph{(f)}, the map $f$ is affinely generating, so part~\emph{(a)} shows that $\Psi_f$ is injective.

Conversely, let
\[
  \Psi\colon H^*\hookrightarrow\ker\overline A_m
\]
be injective. For each $\chi\in H^*$, let $x^\chi$ be the unique representative of $\Psi(\chi)$ whose $v_0$-coordinate is zero. Such a representative is obtained from any representative $x$ by replacing it with $x-x_{v_0}\mathbf1$. Since $\Psi$ is a homomorphism and normalization at $v_0$ is additive,
\[
  x^{\chi+\psi}=x^\chi+x^\psi
  \qquad
  (\chi,\psi\in H^*).
\]
For every vertex $v$, define
\[
  \eta_v\colon H^*\longrightarrow\Z/m\Z,
  \qquad
  \eta_v(\chi)=x^\chi_v.
\]
The preceding equality shows that $\eta_v$ is a homomorphism. Since $m$ is divisible by $\ex(H)$, Lemma~\ref{lem:duality} identifies $H$ with $\operatorname{Hom}(H^*,\Z/m\Z)$ through evaluation. Consequently, for every $v\in V(G)$ there is a unique element $f(v)\in H$ such that
\[
  \chi(f(v))=\eta_v(\chi)=x^\chi_v
  \qquad
  \text{for every }\chi\in H^*.
\]
This defines a map $f\colon V(G)\to H$. Since $x^\chi_{v_0}=0$ for every $\chi\in H^*$ and the characters in $H^*$ separate points of $H$, one has $f(v_0)=0$. The defining equality also gives
\[
  \chi\circ f=x^\chi
  \qquad\text{and hence}\qquad
  [\chi\circ f]=\Psi(\chi)
  \quad
  \text{for every }\chi\in H^*.
\]

For every $\chi\in H^*$, the inclusion $\Psi(\chi)\in\ker\overline A_m$ gives $[A_mx^\chi]=0$, so $A_mx^\chi$ is constant. Moreover, for each vertex $v$,
\[
  (A_mx^\chi)_v
  =\sum_{u\in N_G(v)}\chi(f(u))
  =\chi\left(\sum_{u\in N_G(v)}f(u)\right).
\]
It follows that, for all vertices $v,w$,
\[
  \chi\left(
    \sum_{u\in N_G(v)}f(u)
    -\sum_{u\in N_G(w)}f(u)
  \right)=0
  \qquad
  \text{for every }\chi\in H^*.
\]
Lemma~\ref{lem:duality} shows that the elements of $H^*$ separate points of $H$. Hence the two neighborhood sums are equal, and $f$ is magic.

If $\langle f(V(G))\rangle<H$, Lemma~\ref{lem:duality} gives a nonzero $\chi\in H^*$ vanishing on $\langle f(V(G))\rangle$. Then $\chi\circ f$ is identically zero, so $\Psi(\chi)=0$, contradicting the injectivity of $\Psi$. Therefore $\langle f(V(G))\rangle=H$, and $f$ is generating. This proves that $f\mapsto\Psi_f$ is surjective.

It remains to prove that $f\mapsto\Psi_f$ is injective. Suppose that $f,g\colon V(G)\to H$ are generating $H$-magic maps satisfying $f(v_0)=g(v_0)=0$ and $\Psi_f=\Psi_g$. Then, for every $\chi\in H^*$, the maps $\chi\circ f$ and $\chi\circ g$ differ by a constant. Since both vanish at $v_0$, they are equal. Thus
\[
  \chi(f(v)-g(v))=0
\]
for every $v\in V(G)$ and every $\chi\in H^*$. Lemma~\ref{lem:duality} gives $f=g$, so $f\mapsto\Psi_f$ is injective.
\end{proof}

\begin{proof}[\normalfont\bfseries Proof of Theorem~\ref{thm:main-characterization}]
For part~\emph{(a)}, suppose first that $G$ admits a generating $\Gamma$-magic map. By Lemma~\ref{lem:generating-affine}, there exists $H\le\Gamma$ such that $\Gamma/H$ is cyclic and $G$ admits an affinely generating $H$-magic map. By Proposition~\ref{prop:affine-basic}\emph{(f)}, there is a generating $H$-magic map $g$ with $g(v_0)=0$ for some $v_0\in V(G)$. Since $\ex(H)\mid d_r$, Lemma~\ref{lem:affine-param}\emph{(b)} with $m=d_r$ gives
\[
  \operatorname{Hom}(H,\Z/d_r\Z)
  \hookrightarrow
  \ker\overline A_{d_r}.
\]
Lemma~\ref{lem:duality} identifies the domain with $H$, while Lemma~\ref{lem:cyclic-quotient} gives $\BG\hookrightarrow H$. Hence
\[
  \BG\hookrightarrow H\hookrightarrow\ker\overline A_{d_r}.
\]

Conversely, suppose
\[
  \BG\hookrightarrow\ker\overline A_{d_r}.
\]
Since $\ex(\BG)\mid d_r$, Lemma~\ref{lem:duality} identifies $\operatorname{Hom}(\BG,\Z/d_r\Z)$ with $\BG$. Lemma~\ref{lem:affine-param}\emph{(b)} gives a generating $\BG$-magic map $g$ satisfying $g(v_0)=0$ for a fixed vertex $v_0$. Proposition~\ref{prop:affine-basic}\emph{(f)} shows that $g$ is affinely generating. Under the fixed invariant-factor decomposition, regard $\BG$ as the corresponding subgroup of $\Gamma$. Then $\Gamma/\BG\cong\Z/d_r\Z$. Lemma~\ref{lem:generating-affine} now gives a generating $\Gamma$-magic map.

For part~\emph{(b)}, Proposition~\ref{prop:affine-basic}\emph{(f)} shows that $G$ admits an affinely generating $\Gamma$-magic map if and only if, for a fixed vertex $v_0$, it admits a generating $\Gamma$-magic map $f$ with $f(v_0)=0$. By Lemma~\ref{lem:affine-param}\emph{(b)} with $m=d_r$, this is equivalent to
\[
  \operatorname{Hom}(\Gamma,\Z/d_r\Z)
  \hookrightarrow
  \ker\overline A_{d_r}.
\]
Since $d_r=\ex(\Gamma)$, Lemma~\ref{lem:duality} identifies the domain with $\Gamma$, proving part~\emph{(b)}.

For part~\emph{(c)}, let $\ell$ be a $\Gamma$-distance magic labeling. Choose $v_0\in V(G)$ with $\ell(v_0)=0$. Since $\ell$ is bijective, it is generating, and Lemma~\ref{lem:affine-param}\emph{(b)} gives an embedding
\[
  \Psi_\ell\colon
  \operatorname{Hom}(\Gamma,\Z/d_r\Z)
  \hookrightarrow
  \ker\overline A_{d_r},
  \qquad
  \Psi_\ell(\chi)=[\chi\circ\ell].
\]
Its image $W$ is isomorphic to $\Gamma$. If $v\ne w$, then $\ell(v)-\ell(w)\ne0$. Lemma~\ref{lem:duality} gives $\chi\in\operatorname{Hom}(\Gamma,\Z/d_r\Z)$ such that
\[
  \chi(\ell(v)-\ell(w))\ne0.
\]
Thus the class $[\chi\circ\ell]\in W$ separates $v$ and $w$, so $W$ is vertex-separating.

Conversely, suppose that $W\le\ker\overline A_{d_r}$ is vertex-separating and isomorphic to $\Gamma$. By Lemma~\ref{lem:duality}, choose an isomorphism
\[
  \Psi\colon
  \operatorname{Hom}(\Gamma,\Z/d_r\Z)
  \longrightarrow W.
\]
Fix $v_0\in V(G)$. Lemma~\ref{lem:affine-param}\emph{(b)} gives a generating $\Gamma$-magic map $f$ with $f(v_0)=0$ and
\[
  [\chi\circ f]=\Psi(\chi)
  \qquad
  \text{for every }\chi\in\operatorname{Hom}(\Gamma,\Z/d_r\Z).
\]
If $f(v)=f(w)$ for distinct vertices, then
\[
  x_v-x_w=0
\]
for every $[x]\in W$, where $x$ is any representative of $[x]$. This contradicts vertex separation. Thus $f$ is injective. Since $|V(G)|=|\Gamma|$, it is bijective and hence is a $\Gamma$-distance magic labeling.
\end{proof}

%======================================================================
\section{The reduced Smith formulation}\label{sec:smith}

For an abelian group $C$, write
\[
  C[m]=\{c\in C\mid mc=0\}
\]
for its $m$-torsion subgroup.
For a free abelian group $L$, write $\operatorname{End}_{\Z}(L)$ for the ring of $\Z$-linear endomorphisms of $L$.

\begin{lemma}\label{lem:modular-kernel}
Let $L$ be a free abelian group of finite rank and let $T\in\operatorname{End}_{\Z}(L)$ be injective. Put $C=\operatorname{coker}T$. For every positive integer $m$, if $T_m$ is the induced endomorphism of $L/mL$, then
\[
  \ker T_m\cong C[m].
\]
\end{lemma}

\begin{proof}
Let $t=\operatorname{rank}L$. Choose bases $e_1,\ldots,e_t$ and $f_1,\ldots,f_t$ of $L$ that put $T$ in Smith normal form, so that
\[
  T(e_i)=s_i f_i
  \qquad(1\le i\le t),
\]
where $s_1,\ldots,s_t$ are positive integers satisfying $s_1\mid\cdots\mid s_t$. Then
\[
  C\cong\bigoplus_{i=1}^t\Z/s_i\Z.
\]
For $1\le i\le t$, let
\[
  S_i\colon\Z/m\Z\longrightarrow\Z/m\Z,
  \qquad
  S_i([a])=[s_i a],
\]
be multiplication by $s_i$ modulo $m$.
Write $x=\sum_{i=1}^t x_i e_i$. The class of $x$ lies in $\ker T_m$ if and only if $Tx\in mL$. In the basis $f_1,\ldots,f_t$, this is equivalent to
\[
  m\mid s_i x_i
  \qquad(1\le i\le t).
\]
Thus
\[
  \ker T_m
  \cong
  \bigoplus_{i=1}^t
  \ker S_i.
\]
Fix $i$ and put $g_i=\gcd(m,s_i)$, $m=g_i m_i$, and $s_i=g_i s_i'$. Since $\gcd(m_i,s_i')=1$, one has
\[
  [a]\in\ker S_i
  \quad\Longleftrightarrow\quad
  m\mid s_i a
  \quad\Longleftrightarrow\quad
  m_i\mid a.
\]
Consequently,
\[
  \ker S_i=\langle[m_i]\rangle
\]
is cyclic of order $g_i$. On the other hand, write $[b]_{s_i}$ for the residue class of an integer $b$ modulo $s_i$. Then
\[
  (\Z/s_i\Z)[m]
  =\{[b]_{s_i}\mid s_i\mid mb\}
  =\langle[s_i']_{s_i}\rangle,
\]
which is also cyclic of order $g_i$. The assignment $[m_i]\mapsto[s_i']_{s_i}$ therefore extends to an isomorphism
\[
  \ker S_i\cong(\Z/s_i\Z)[m].
\]
Applying these isomorphisms to all coordinates gives
\[
  \ker T_m
  \cong
  \bigoplus_{i=1}^t(\Z/s_i\Z)[m]
  \cong C[m].
\]
\end{proof}

\begin{proposition}\label{prop:smith-exact}
Let $G$ be a $k$-regular graph with $k>0$. If $\overline A$ is nonsingular over $\mathbb Q$, then $A$ is nonsingular over $\mathbb Q$. If $[\mathbf1]_A$ denotes the class of $\mathbf1$ in $\operatorname{coker}A$, then
\[
  \Sred(G)\cong
  \frac{\operatorname{coker}A}{\langle[\mathbf1]_A\rangle},
  \qquad
  \langle[\mathbf1]_A\rangle\cong\Z/k\Z.
\]
In particular,
\[
  |\Sred(G)|=\frac{|\det A|}{k}.
\]
\end{proposition}

\begin{proof}
We first prove that $A$ is nonsingular. Since $\overline A$ is nonsingular over $\mathbb Q$, it is injective on the free abelian group $\Lamone$. If $Ax=0$ for $x\in\Z^{V(G)}$, then $\overline A[x]=0$, so $[x]=0$ and therefore $x=t\mathbf1$ for some $t\in\Z$. Since
\[
  0=Ax=tk\mathbf1
\]
and $k>0$, one has $t=0$. Thus $A$ is injective on $\Z^{V(G)}$.

If $A$ were singular over $\mathbb Q$, there would be a nonzero vector $z\in\mathbb Q^{V(G)}$ with $Az=0$. Multiplying $z$ by a common denominator of its coordinates would produce a nonzero vector in $\Z^{V(G)}$ annihilated by $A$, contradicting injectivity. Hence $A$ is nonsingular over $\mathbb Q$.

For $x\in\Z^{V(G)}$, write $[x]_A$ for its class in $\operatorname{coker}A$ and $[x]_{\mathbf1}$ for its class in $\Lamone$. The assignment
\[
  \Phi\colon\operatorname{coker}A\twoheadrightarrow\Sred(G),
  \qquad
  \Phi([x]_A)=[x]_{\mathbf1}+\overline A(\Lamone),
\]
defines a surjective homomorphism. It is well defined because $[Az]_{\mathbf1}=\overline A([z]_{\mathbf1})$ for every $z\in\Z^{V(G)}$, and it is surjective because every class in $\Sred(G)$ has a representative in $\Lamone$ of the form $[x]_{\mathbf1}$.

If $[x]_A\in\ker\Phi$, then
\[
  [x]_{\mathbf1}=\overline A([y]_{\mathbf1})=[Ay]_{\mathbf1}
\]
for some $y\in\Z^{V(G)}$. Hence
\[
  x-Ay=t\mathbf1
\]
for some $t\in\Z$, so $[x]_A=t[\mathbf1]_A$. Conversely, since $\mathbf1\in\Z\mathbf1$, its class in $\Lamone$ is zero. Therefore
\[
  \Phi([\mathbf1]_A)
  = [\mathbf1]_{\mathbf1}+\overline A(\Lamone)
  =0
  \qquad\text{in }\Sred(G),
\]
and hence $[\mathbf1]_A\in\ker\Phi$. Thus
\[
  \ker\Phi=\langle[\mathbf1]_A\rangle.
\]

Since $A\mathbf1=k\mathbf1$, this class has order dividing $k$. Conversely, if $t\mathbf1=Ay$ for some $y\in\Z^{V(G)}$, then, because $A$ is nonsingular and $A^{-1}\mathbf1=k^{-1}\mathbf1$ over $\mathbb Q$, one has
\[
  y=\frac{t}{k}\mathbf1.
\]
The integrality of $y$ implies $k\mid t$. Hence the class of $\mathbf1$ has order exactly $k$, and therefore
\[
  \ker\Phi\cong\Z/k\Z.
\]
The First Isomorphism Theorem now gives
\[
  \Sred(G)\cong
  \frac{\operatorname{coker}A}{\langle[\mathbf1]_A\rangle}.
\]

Finally, put $n=|V(G)|$. The Smith normal form theorem applied to the nonsingular integer matrix $A$ gives positive integers $a_1,\ldots,a_n$ such that
\[
  \operatorname{coker}A\cong
  \bigoplus_{i=1}^n\Z/a_i\Z
  \qquad\text{and}\qquad
  |\det A|=\prod_{i=1}^n a_i.
\]
Consequently,
\[
  |\operatorname{coker}A|=|\det A|,
\]
and the preceding quotient isomorphism together with $|\langle[\mathbf1]_A\rangle|=k$ yields
\[
  |\Sred(G)|=\frac{|\det A|}{k}.
\]
\end{proof}

\begin{proof}[\normalfont\bfseries Proof of Corollary~\ref{cor:main-smith}]
Let $d_r=\ex(\Gamma)$. There is a canonical identification
\[
  \Lamone/d_r\Lamone
  \cong
  \frac{(\Z/d_r\Z)^{V(G)}}{(\Z/d_r\Z)\mathbf1},
\]
under which $\overline A_{d_r}$ is the reduction of $\overline A$ on $\Lamone/d_r\Lamone$. Hence Lemma~\ref{lem:modular-kernel}, applied with $L=\Lamone$ and $T=\overline A$, gives
\[
  \ker\overline A_{d_r}\cong\Sred(G)[d_r].
\]
Since $d_rH=0$ for $H=\Gamma$ and $H=\BG$, one has
\[
  H\hookrightarrow\Sred(G)[d_r]
  \quad\Longleftrightarrow\quad
  H\hookrightarrow\Sred(G).
\]
Parts~\emph{(a)} and \emph{(b)} therefore follow from the corresponding parts of Theorem~\ref{thm:main-characterization}.

For the numerical consequences, Proposition~\ref{prop:smith-exact} gives
\[
  |\Sred(G)|=\frac{|\det A|}{k}.
\]
If $H$ is a subgroup of a finite abelian group $K$, then Lagrange's theorem gives $|H|\mid|K|$. Moreover, the order of every element of $H$ divides $\ex(K)$. Since $\ex(H)$ is the least common multiple of these element orders, it follows that $\ex(H)\mid\ex(K)$. Apply these facts to the embeddings of $\Gamma$ and $\BG$ into $\Sred(G)$. Since $\ex(\Gamma)=d_r$ and, for $r\ge2$,
\[
  \ex(\BG)=d_{r-1},
  \qquad
  |\BG|=\frac{|\Gamma|}{d_r},
\]
the stated divisibilities follow from the two subgroup embeddings.

Finally, if $G$ admits a $\Gamma$-distance magic labeling, then Theorem~\ref{thm:main-characterization}\emph{(c)} gives a vertex-separating copy of $\Gamma$ in $\ker\overline A_{d_r}$. Under the preceding identification this yields $\Gamma\hookrightarrow\Sred(G)$, and the numerical consequences are the same as in part~\emph{(b)}.
\end{proof}

\subsection{Examples of graphs which admit affinely generating magic maps but no $\Gamma$-distance magic labelings}\label{subsec:strict-affine}

The additional vertex-separation condition in Theorem~\ref{thm:main-characterization}\emph{(c)} is essential. We first give a disconnected example in which the graph and the labeling group have the same order.

Let $T_1,\ldots,T_5$ be pairwise vertex-disjoint triangles, let $C_{17}$ be a $17$-cycle disjoint from them, and put
\[
  G=C_{17}\mathbin{\sqcup}T_1\mathbin{\sqcup}\cdots\mathbin{\sqcup}T_5,
  \qquad
  \Gamma=(\Z/2\Z)^5.
\]
Let $e_1,\ldots,e_5$ be the standard basis of $\Gamma$, and define $f\colon V(G)\to\Gamma$ by
\[
  f(v)=
  \begin{cases}
    0,   & v\in V(C_{17}),\\
    e_i, & v\in V(T_i),\quad 1\le i\le5.
  \end{cases}
\]
The graph has order $17+5\cdot3=32=|\Gamma|$. The map $f$ is an affinely generating $\Gamma$-magic map with magic constant $0$, whereas $G$ admits no $\Gamma$-distance magic labeling. The verification of these assertions is left to the reader.

Let
\[
  \Gamma=(\Z/2\Z)^3.
\]
The following is a 2-connected example of a graph admitting an affinely generating $\Gamma$-magic map but no $\Gamma$-distance magic labeling. Let $e_1,e_2,e_3$ be the standard basis of $\Gamma$. The all-one vector in $\Gamma$ is
\[
  \mathbf1=e_1+e_2+e_3.
\]
For $x\in\{e_1,e_2,e_3\}$, write $\bar x=x+\mathbf1$. Consider the graph and the $\Gamma$-valued labeling displayed in Figure~\ref{fig:two-pentagons}.

\begin{figure}[ht]
\centering
\begin{tikzpicture}[
  vertex/.style={circle,fill=black,inner sep=2.4pt},
  every label/.style={font=\LARGE}
]
  \node[vertex,label=above:{$\bar e_1$}] (u) at (0,1.25) {};
  \node[vertex,label=below:{$\bar e_1$}] (v) at (0,-1.25) {};

  \node[vertex,label=above:{$e_2$}] (lu) at (-1.75,2.05) {};
  \node[vertex,label=left:{$e_1$}] (lm) at (-3,0) {};
  \node[vertex,label=below left:{$\bar e_2$}] (ll) at (-1.75,-2.05) {};

  \node[vertex,label=above:{$\bar e_3$}] (ru) at (1.75,2.05) {};
  \node[vertex,label=right:{$e_1$}] (rm) at (3,0) {};
  \node[vertex,label=below right:{$e_3$}] (rl) at (1.75,-2.05) {};

  \draw[thick]
    (u)--(lu)--(lm)--(ll)--(v)--(u)
    (u)--(ru)--(rm)--(rl)--(v);
\end{tikzpicture}
\caption{A $\Gamma$-magic map on two $5$-cycles sharing an edge.}
\label{fig:two-pentagons}
\end{figure}
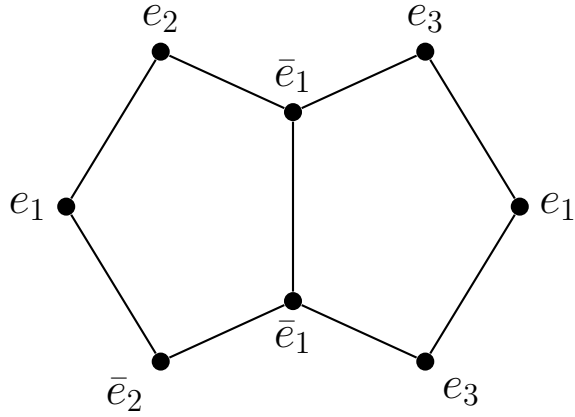

The graph has order $8=|\Gamma|$ and vertex-connectivity $2$. The displayed labeling is an affinely generating $\Gamma$-magic map with magic constant $\mathbf1$, whereas the graph admits no $\Gamma$-distance magic labeling. The verification of these assertions is left to the reader.

%======================================================================
\section{A counterexample and the generating group distance magic conjecture}\label{sec:weakened-conjecture}

The following proposition collects the modular criteria that will be used in this section.

\begin{proposition}\label{prop:group-distance-kernel-test}
Let $G$ be a regular graph of order $N$.
For a prime $p$, let $v_p(N)$ denote the $p$-adic valuation of $N$.
\begin{enumerate}[label=\emph{(\alph*)},leftmargin=2.2em]
\item Let $p$ be a prime, $r\ge2$, and $\Gamma=(\Z/p\Z)^r$. Then $G$ admits a generating $\Gamma$-magic map if and only if
\[
  \dim_{\F_p}\ker\overline A_p\ge r-1,
\]
and it admits an affinely generating $\Gamma$-magic map if and only if
\[
  \dim_{\F_p}\ker\overline A_p\ge r.
\]
\item If $G$ is group distance magic, then, for every prime $p\mid N$,
\[
  \dim_{\F_p}\ker\overline A_p\ge v_p(N).
\]
\end{enumerate}
\end{proposition}

\begin{proof}
For part~\emph{(a)}, one has $d_r=p$ and
\[
  \BG\cong(\Z/p\Z)^{r-1}.
\]
The group $\ker\overline A_p$ is an $\F_p$-vector space, and an embedding of $(\Z/p\Z)^s$ into this kernel exists exactly when its dimension is at least $s$. The two assertions therefore follow from parts~\emph{(a)} and \emph{(b)} of Theorem~\ref{thm:main-characterization}.

For part~\emph{(b)}, fix $p\mid N$ and put $a=v_p(N)$. Choose an abelian group $\Gamma$ of order $N$ whose $p$-primary component is $(\Z/p\Z)^a$. If $G$ is group distance magic, then it admits a $\Gamma$-distance magic labeling, which is affinely generating. Lemma~\ref{lem:affine-param}\emph{(a)}, applied with modulus $p$, gives an injection
\[
  \operatorname{Hom}(\Gamma,\F_p)
  \hookrightarrow
  \ker\overline A_p.
\]
The domain has dimension $a$, which proves the claimed inequality.
\end{proof}

Cichacz and Froncek conjectured that every distance magic graph is group distance magic~\cite[Conjecture~29]{CF16}; the conjecture was later restated in~\cite[Conjecture~1.7]{CDF20}. Theorem~\ref{thm:main-characterization} separates the bijectivity requirement from two successively weaker generation conditions. A first natural weakening of this conjecture is the assertion that every distance magic graph $G$ admits an affinely generating $\Gamma$-magic map for every abelian group $\Gamma$ of order $|V(G)|$. The following theorem shows that even this weaker assertion is false.

\begin{theorem}\label{thm:affine-counterexample}
There exists a $6$-regular distance magic graph $G$ of order $27$ such that
\[
  \dim_{\F_3}\ker\overline A_3=2.
\]
Consequently, for $\Gamma=(\Z/3\Z)^3$, the graph $G$ admits a generating $\Gamma$-magic map but no affinely generating $\Gamma$-magic map. In particular, $G$ is not $\Gamma$-distance magic and hence is not group distance magic.
\end{theorem}

\begin{proof}
Let $V(G)=\{1,2,\ldots,27\}$, and define $G$ by the following neighborhoods.
\begingroup
\scriptsize
\[
\renewcommand{\arraystretch}{1.15}
\begin{array}{c|l@{\quad}c|l@{\quad}c|l}
\toprule
v & N_G(v) & v & N_G(v) & v & N_G(v)\\
\midrule
1 & \{7,9,11,14,18,25\} & 10 & \{3,5,12,13,25,26\} & 19 & \{3,4,7,21,22,27\}\\
2 & \{8,11,12,13,17,23\} & 11 & \{1,2,16,17,22,26\} & 20 & \{3,9,14,16,18,24\}\\
3 & \{5,6,10,19,20,24\} & 12 & \{2,6,10,17,24,25\} & 21 & \{7,8,9,16,19,25\}\\
4 & \{7,8,9,15,19,26\} & 13 & \{2,7,10,15,23,27\} & 22 & \{5,9,11,17,19,23\}\\
5 & \{3,6,10,16,22,27\} & 14 & \{1,6,15,17,20,25\} & 23 & \{2,6,13,17,22,24\}\\
6 & \{3,5,12,14,23,27\} & 15 & \{4,8,13,14,18,27\} & 24 & \{3,8,12,18,20,23\}\\
7 & \{1,4,13,19,21,26\} & 16 & \{5,9,11,18,20,21\} & 25 & \{1,10,12,14,21,26\}\\
8 & \{2,4,15,18,21,24\} & 17 & \{2,11,12,14,22,23\} & 26 & \{4,7,10,11,25,27\}\\
9 & \{1,4,16,20,21,22\} & 18 & \{1,8,15,16,20,24\} & 27 & \{5,6,13,15,19,26\}\\
\bottomrule
\end{array}
\]
\endgroup
The table defines a simple $6$-regular graph, shown in Figure~\ref{fig:affine-counterexample}. Define
\[
  \ell(v)=v
  \qquad (v\in V(G)).
\]
For every vertex $v$, the six integers in $N_G(v)$ have sum $84$. Hence $\ell$ is an ordinary distance magic labeling of $G$ with magic constant $84$.

Let $A=(a_{ij})$ be the adjacency matrix of $G$. The classes $[e_1],\ldots,[e_{26}]$ form a basis of
\[
  \Lamone=\Z^{27}/\Z\mathbf1.
\]
Since
\[
  [e_{27}]=-\sum_{i=1}^{26}[e_i],
\]
for $1\le j\le26$ we have
\[
  \overline A([e_j])
  =[Ae_j]
  =\sum_{i=1}^{26}(a_{ij}-a_{27,j})[e_i].
\]
Hence, with respect to this basis, the reduced adjacency operator is represented by
\[
  M=(a_{ij}-a_{27,j})_{1\le i,j\le26}.
\]
The matrix $M$ is displayed explicitly in Appendix~\ref{app:matrix}. Let $M'$ be the $24\times24$ submatrix of $M$ formed by rows $3,\ldots,26$ and columns $1,\ldots,24$. A direct determinant computation gives
\[
  \det M'=-11.
\]
Since $-11\not\equiv0\pmod 3$, the reduction of $M'$ modulo $3$ is nonsingular. Hence
\[
  \operatorname{rank}_{\F_3}\overline A_3\ge24
\]
and therefore
\[
  \dim_{\F_3}\ker\overline A_3\le2.
\]

For the reverse inequality, consider the vectors
\[
  u=(1,2,0,1,2,0,1,2,0,1,2,0,1,2,0,1,2,0,1,2,0,1,2,0,1,2)^T
\]
and
\[
  v=(1,1,2,0,2,0,2,1,1,2,0,0,2,0,2,0,1,1,0,0,0,0,0,0,0,1)^T
\]
in $\F_3^{26}$. With respect to the basis $[e_1],\ldots,[e_{26}]$, the first vector is
\[
  u=\bigl(\ell(1)-\ell(27),\ldots,\ell(26)-\ell(27)\bigr)^T\pmod 3,
\]
so it is precisely the coordinate vector of the reduction modulo $3$ of the ordinary distance magic labeling $\ell(i)=i$. Hence $Mu=0$ over $\F_3$. Direct multiplication also gives $Mv=0$ over $\F_3$. The last two coordinates of $u$ and $v$ are respectively $(1,2)$ and $(0,1)$, so these vectors are linearly independent. Therefore
\[
  \dim_{\F_3}\ker\overline A_3\ge2.
\]
Consequently,
\[
  \dim_{\F_3}\ker\overline A_3=2.
\]

Take $\Gamma=(\Z/3\Z)^3$. Proposition~\ref{prop:group-distance-kernel-test}\emph{(a)} now gives both conclusions. Since the nullity is $2=r-1$, the graph admits a generating $\Gamma$-magic map, while an affinely generating $\Gamma$-magic map would require nullity at least $3$. Hence $G$ is not $\Gamma$-distance magic and therefore is not group distance magic.
\end{proof}

\begin{figure}[htbp]
\centering
\begin{tikzpicture}[scale=0.86]
  \foreach \i in {1,...,27}{
    \coordinate (c\i) at ({90 - (\i-1)*13.3333}:5.6cm);
  }
  \foreach \a/\b in {1/7,1/9,1/11,1/14,1/18,1/25,2/8,2/11,2/12,2/13,2/17,2/23,
                     3/5,3/6,3/10,3/19,3/20,3/24,4/7,4/8,4/9,4/15,4/19,4/26,
                     5/6,5/10,5/16,5/22,5/27,6/12,6/14,6/23,6/27,7/13,7/19,
                     7/21,7/26,8/15,8/18,8/21,8/24,9/16,9/20,9/21,9/22,10/12,
                     10/13,10/25,10/26,11/16,11/17,11/22,11/26,12/17,12/24,
                     12/25,13/15,13/23,13/27,14/15,14/17,14/20,14/25,15/18,
                     15/27,16/18,16/20,16/21,17/22,17/23,18/20,18/24,19/21,
                     19/22,19/27,20/24,21/25,22/23,23/24,25/26,26/27}{
    \draw[gray!70,line width=0.3pt] (c\a) -- (c\b);
  }
  \foreach \i in {1,...,27}{
    \node[circle,draw,fill=white,inner sep=0pt,minimum size=4.6mm,font=\tiny]
      at (c\i) {\i};
  }
\end{tikzpicture}
\caption{The graph $G$ of Theorem~\ref{thm:affine-counterexample}.}
\label{fig:affine-counterexample}
\end{figure}
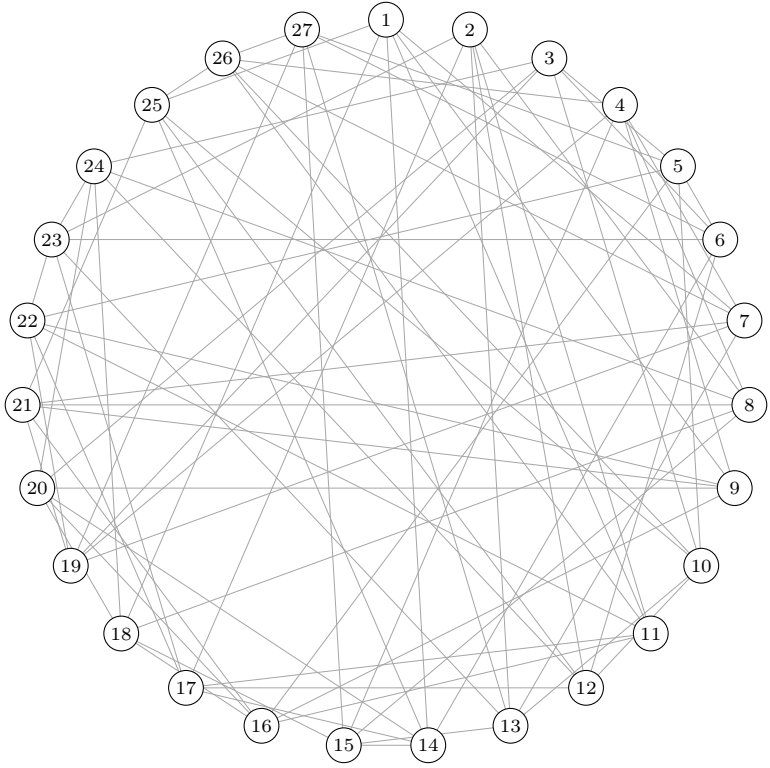

\begin{remark}\label{rem:affine-counterexample-computation}
The graph in Theorem~\ref{thm:affine-counterexample} was found computationally. We fixed the vertex set $\{1,\ldots,27\}$ and the labeling $\ell(i)=i$, and searched for a $6$-regular graph by means of a $0$--$1$ integer feasibility problem. For each unordered pair $\{i,j\}$, let $x_{ij}=x_{ji}\in\{0,1\}$ be the indicator of the edge $ij$. The constraints were
\[
  \sum_{j\ne i}x_{ij}=6
  \qquad\text{and}\qquad
  \sum_{j\ne i}j x_{ij}=84
  \qquad (1\le i\le27).
\]
Thus every feasible solution is a $6$-regular graph for which the identity labeling is distance magic. Feasible solutions were then tested by computing the rank of the reduced adjacency matrix modulo $3$. For the graph above this rank is exactly $24$. The proof gives an exact certificate for this value without relying on the rank computation. The $24\times24$ minor with determinant $-11$ gives rank at least $24$, while the two explicit linearly independent null vectors give nullity at least $2$ and hence rank at most $24$. Appendix~\ref{app:verification} gives a short verification script that reconstructs the adjacency matrix from the neighborhood table, checks regularity and the distance-magic equations, forms the reduced adjacency matrix, verifies the two null vectors and the determinant certificate, and independently computes its ranks over $\F_3$ and $\mathbb Q$.
\end{remark}

Theorem~\ref{thm:affine-counterexample} shows that affine generation remains too strong for a universal consequence of ordinary distance magic. This leads to the following generating group distance magic conjecture.

\begin{conjecture}\label{conj:generating-distance-magic}
Let $G$ be a distance magic graph of order $n$. Then, for every abelian group $\Gamma$ of order $n$, the graph $G$ admits a generating $\Gamma$-magic map.
\end{conjecture}

For regular distance magic graphs of nonsquarefree order, the following proposition gives a structural consequence in the direction of Conjecture~\ref{conj:generating-distance-magic}. It shows that generating magic maps cannot occur only over cyclic groups.

\begin{proposition}\label{prop:cyclic-only-affine}
Let $G$ be a regular graph of order $n$, where $p^2\mid n$ for some prime $p$. Suppose that, for every abelian group $\Gamma$ of order $n$, the existence of a generating $\Gamma$-magic map implies that $\Gamma$ is cyclic. Then $G$ admits no affinely generating $\Gamma$-magic map and no $\Gamma$-distance magic labeling for any abelian group $\Gamma$ of order $n$.
\end{proposition}

\begin{proof}
Suppose that $G$ admits an affinely generating $\Gamma$-magic map for some abelian group $\Gamma$ of order $n$. Every affinely generating map is generating, so the hypothesis forces $\Gamma\cong\Z/n\Z$. Fix a prime $p$ with $p^2\mid n$ and put $m=n/p$. Since $m\mid n$, Lemma~\ref{lem:affine-param}\emph{(a)} gives
\[
  \Z/m\Z\cong\operatorname{Hom}(\Z/n\Z,\Z/m\Z)\hookrightarrow\ker\overline A_m,
\]
and $p\mid m$, so $\Z/p\Z\hookrightarrow\ker\overline A_m$. Now
\[
  \Gamma'=\Z/p\Z\oplus\Z/m\Z
\]
is an invariant-factor decomposition because $p\mid m$. The group $\Gamma'$ has order $n$, exponent $m$ and $B_{\Gamma'}\cong\Z/p\Z$, and it is noncyclic. Theorem~\ref{thm:main-characterization}\emph{(a)} therefore produces a generating $\Gamma'$-magic map, contradicting the hypothesis. The assertion about distance magic labelings follows because these are affinely generating.
\end{proof}

Let $\ell\colon V(G)\to\{1,\ldots,n\}$ be an ordinary distance magic labeling of a regular graph $G$ of order $n$. Reduction modulo $n$ turns $\ell$ into a bijective $\Z/n\Z$-magic map and hence into an affinely generating $\Z/n\Z$-magic map. If $p^2\mid n$ for some prime $p$, the contrapositive of Proposition~\ref{prop:cyclic-only-affine} therefore shows that $G$ admits a generating $\Gamma$-magic map for some noncyclic abelian group $\Gamma$ of order $n$.

A stronger positive result holds for all labeling groups with at most two generators.

\begin{theorem}\label{thm:two-generated-distance-magic}
Let $G$ be a regular distance magic graph of order $n$, and let $\Gamma$ be an abelian group of order $n$ generated by at most two elements. Then $G$ admits a generating $\Gamma$-magic map.
\end{theorem}

\begin{proof}
If $\Gamma$ is cyclic, a constant map whose value is a generator of $\Gamma$ is generating and magic because $G$ is regular. Suppose that $\Gamma$ is noncyclic. Its invariant factor decomposition has the form
\[
  \Gamma\cong \Z/d_1\Z\oplus\Z/d_2\Z,
  \qquad d_1\mid d_2,
\]
with $d_1>1$. Let
\[
  \ell\colon V(G)\longrightarrow\{1,\ldots,n\}
\]
be an ordinary distance magic labeling, with integer magic constant $c$. Define
\[
  g(v)=\ell(v)\pmod{d_2}.
\]
Reducing the neighborhood-sum equations modulo $d_2$ shows that $g$ is a $\Z/d_2\Z$-magic map. Since $\ell$ is bijective, there are vertices $x,y$ with $\ell(x)=1$ and $\ell(y)=2$. Hence
\[
  g(y)-g(x)=1\pmod{d_2},
\]
so $D(g)=\Z/d_2\Z$. Thus $g$ is affinely generating. Theorem~\ref{thm:main-characterization}\emph{(b)} gives
\[
  \Z/d_2\Z\hookrightarrow\ker\overline A_{d_2}.
\]
Since $d_1\mid d_2$, one has
\[
  \Z/d_1\Z\hookrightarrow\Z/d_2\Z
  \hookrightarrow\ker\overline A_{d_2}.
\]
For the present invariant factor decomposition, $B_\Gamma\cong\Z/d_1\Z$. Theorem~\ref{thm:main-characterization}\emph{(a)} therefore yields a generating $\Gamma$-magic map on $G$.
\end{proof}

In particular, the graph in Theorem~\ref{thm:affine-counterexample} is not a counterexample to Conjecture~\ref{conj:generating-distance-magic}. Up to isomorphism, the abelian groups of order $27$ are $\Z/27\Z$, $\Z/9\Z\oplus\Z/3\Z$, and $(\Z/3\Z)^3$. The first two are covered by Theorem~\ref{thm:two-generated-distance-magic}, while the last admits a generating magic map by Theorem~\ref{thm:affine-counterexample}.

\begin{corollary}\label{cor:cubefree-distance-magic}
Let $G$ be a regular distance magic graph of cube-free order $n$, that is, $p^3\nmid n$ for every prime $p$. Then, for every abelian group $\Gamma$ of order $n$, the graph $G$ admits a generating $\Gamma$-magic map.
\end{corollary}

\begin{proof}
Let
\[
  \Gamma\cong \Z/d_1\Z\oplus\cdots\oplus\Z/d_r\Z,
  \qquad d_1\mid\cdots\mid d_r,
\]
be the invariant factor decomposition, with all displayed factors nontrivial. If $r\ge3$, then any prime divisor $p$ of $d_1$ divides $d_1,d_2,d_3$, and hence $p^3\mid|\Gamma|=n$. This contradicts the assumption that $n$ is cube-free. Thus $r\le2$, so $\Gamma$ is generated by at most two elements. Theorem~\ref{thm:two-generated-distance-magic} applies.
\end{proof}

%======================================================================
\section{Applications to graph families}\label{sec:applications}

In this section, we study the existence of generating, affinely generating, and $\Gamma$-distance magic maps on several standard graph families.

\subsection{Complete graphs}

\begin{proposition}\label{prop:complete}
Let $n\ge2$, and let $\Gamma$ be a finite abelian group. Every $\Gamma$-magic map on $K_n$ is constant. Consequently, $K_n$ admits an affinely generating $\Gamma$-magic map if and only if $\Gamma$ is trivial, and it admits a generating $\Gamma$-magic map if and only if $\Gamma$ is cyclic. In particular, if $|\Gamma|=n$, then $K_n$ admits no $\Gamma$-distance magic labeling.
\end{proposition}

\begin{proof}
Let
\[
  s=\sum_{v\in V(K_n)}f(v).
\]
The neighborhood sum at a vertex $v$ is $s-f(v)$. If these sums are equal for all vertices, then $s-f(v)=s-f(w)$ for every $v,w$, and hence $f(v)=f(w)$. Thus every magic map on $K_n$ is constant.

A constant map has affine difference group $0$, so it is affinely generating if and only if $\Gamma$ is trivial. A constant map with value $a$ is generating exactly when $\langle a\rangle=\Gamma$, which is possible exactly when $\Gamma$ is cyclic. Finally, if $|\Gamma|=n\ge2$, then a constant map is not bijective and hence cannot be a $\Gamma$-distance magic labeling.
\end{proof}

\subsection{Cayley graphs on finite $p$-groups}

For a finite group $P$ and an inverse-closed set $S\subseteq P\setminus\{1\}$, the Cayley graph $\operatorname{Cay}(P,S)$ has vertex set $P$, with $x$ adjacent to $y$ if and only if $x^{-1}y\in S$. The set $S$ is called the \emph{connection set}, and $\operatorname{Cay}(P,S)$ is $|S|$-regular.

Let $K[P]$ denote the group algebra of $P$ over a field $K$. The augmentation map is
\[
  \varepsilon\colon K[P]\longrightarrow K,
  \qquad
  \varepsilon\left(\sum_{g\in P}a_g g\right)=\sum_{g\in P}a_g.
\]
For $a\in K[P]$, the scalar $\varepsilon(a)$ is called the augmentation of $a$. Its kernel is the augmentation ideal.

\begin{lemma}\label{lem:augmentation}
Let $P$ be a finite $p$-group. An element of $\F_p[P]$ is a unit if and only if its augmentation is nonzero.
\end{lemma}

\begin{proof}
It is a classical result that the augmentation ideal of $\F_p[P]$ is nilpotent (see~\cite[Proposition~1.2]{Carlson96}). Let $I$ denote this ideal. If $a\in\F_p[P]$ has nonzero augmentation $\lambda$, then $\lambda\in\F_p^\times$ and hence $\lambda^{-1}$ is defined. Since
\[
  \varepsilon(\lambda^{-1}a)=\lambda^{-1}\varepsilon(a)=1
  \qquad\text{and}\qquad
  \varepsilon(1)=1,
\]
we have
\[
  \varepsilon(\lambda^{-1}a-1)=0.
\]
Thus $\lambda^{-1}a-1\in\ker\varepsilon=I$. Writing
\[
  x=\lambda^{-1}a-1,
\]
we obtain
\[
  \lambda^{-1}a=1+x
  \qquad\text{for some }x\in I.
\]
Since $I$ is nilpotent and $x\in I$, there exists $N\ge1$ such that $x^N=0$. Therefore
\[
  (1+x)\bigl(1-x+x^2-\cdots+(-1)^{N-1}x^{N-1}\bigr)
  =1+(-1)^{N-1}x^N
  =1,
\]
so $1+x$ is a unit. Hence $a$ is a unit. Conversely, suppose that $a$ is a unit and choose $b\in\F_p[P]$ such that $ab=1$. Since $\varepsilon$ is a ring homomorphism,
\[
  \varepsilon(a)\varepsilon(b)=\varepsilon(ab)=\varepsilon(1)=1.
\]
Thus $\varepsilon(a)$ is a unit of $\F_p$, and in particular $\varepsilon(a)\ne0$.
\end{proof}

\begin{lemma}\label{lem:cayley-adjacency}
Let $P$ be a finite group, let $K$ be a field, and let $S\subseteq P\setminus\{1\}$ be inverse-closed. Identify $K^P$ with $K[P]$ by
\[
  f\longmapsto \widehat f:=\sum_{x\in P}f(x)x,
\]
and put
\[
  \sigma_S=\sum_{s\in S}s.
\]
If $A$ is the adjacency operator of $\operatorname{Cay}(P,S)$, then
\[
  \widehat{Af}=\widehat f\,\sigma_S
\]
for every $f\in K^P$. Thus, under this identification, $A$ is right multiplication by $\sigma_S$.
\end{lemma}

\begin{proof}
For $y\in P$, the coefficient of $y$ in $\widehat f\,\sigma_S$ is
\[
  \sum_{s\in S}f(ys^{-1}).
\]
Since $S=S^{-1}$, this equals
\[
  \sum_{s\in S}f(ys).
\]
The neighbors of $y$ in $\operatorname{Cay}(P,S)$ are precisely the vertices $ys$ with $s\in S$. Hence the last sum is $(Af)(y)$. Therefore the coefficient of $y$ in $\widehat f\,\sigma_S$ equals the coefficient of $y$ in $\widehat{Af}$ for every $y\in P$, and the asserted identity follows.
\end{proof}

\begin{theorem}\label{thm:cayley-p}
Let $P$ be a nontrivial finite $p$-group and let
\[
  G=\operatorname{Cay}(P,S)
\]
be a simple undirected Cayley graph of regularity $k=|S|$. Suppose $p\nmid k$. For every abelian group $\Gamma$ of order $|P|$, the graph $G$ admits no affinely generating $\Gamma$-magic map and no $\Gamma$-distance magic labeling. It admits a generating $\Gamma$-magic map if and only if $\Gamma$ is cyclic.
\end{theorem}

\begin{proof}
By Lemma~\ref{lem:cayley-adjacency}, after identifying $\F_p^P$ with $\F_p[P]$, the adjacency operator $A$ is right multiplication by
\[
  \sigma_S=\sum_{s\in S}s.
\]
Its augmentation is $k$, which is nonzero in $\F_p$. Lemma~\ref{lem:augmentation} shows that $\sigma_S$ is a unit. Right multiplication by $\sigma_S^{-1}$ is therefore the inverse of right multiplication by $\sigma_S$. By Lemma~\ref{lem:cayley-adjacency}, this implies that $A$ is invertible over $\F_p$.

Let $C=\F_p\mathbf 1$ be the subspace of constant functions. Since $G$ is $k$-regular,
\[
  A\mathbf 1=k\mathbf 1,
\]
so $A(C)\subseteq C$, and in fact $A$ acts on $C$ as multiplication by $k$. Because $p\nmid k$, the scalar $k$ is invertible in $\F_p$, and hence $A(C)=C$. We now verify directly that the induced map
\[
  \overline A_p\colon \F_p^P/C\longrightarrow \F_p^P/C
\]
is invertible. Suppose that $\overline A_p([x])=0$. Then $Ax=c\mathbf 1$ for some $c\in\F_p$. Since $k$ is invertible,
\[
  A\bigl(c k^{-1}\mathbf 1\bigr)=c\mathbf 1=Ax.
\]
The invertibility of $A$ gives $x=c k^{-1}\mathbf 1\in C$, and therefore $[x]=0$. Thus $\overline A_p$ is injective, and since $\F_p^P/C$ is finite-dimensional, $\overline A_p$ is invertible.

Let $M$ be an integral matrix representing $\overline A$ with respect to a basis of $\Lamone$. Its reduction modulo $p$ represents $\overline A_p$. Since $\overline A_p$ is invertible,
\[
  \det M\not\equiv0\pmod p.
\]
Thus $\det M$ is a unit modulo every power of $p$, and the reduction of $M$ modulo $p^a$ is invertible for every $a\ge1$.

Since $|P|$ is a power of $p$, every abelian group $\Gamma$ of order $|P|$ is a nontrivial $p$-group. Write
\[
  \Gamma\cong\Z/d_1\Z\oplus\cdots\oplus\Z/d_r\Z,
  \qquad d_1\mid\cdots\mid d_r.
\]
Then $d_r=p^a$ for some $a\ge1$, and the reduction of $M$ modulo $d_r$ represents $\overline A_{d_r}$. Hence
\[
  \ker\overline A_{d_r}=0.
\]
Theorem~\ref{thm:main-characterization}\emph{(b)} now excludes affinely generating $\Gamma$-magic maps, and therefore also $\Gamma$-distance magic labelings. By Theorem~\ref{thm:main-characterization}\emph{(a)}, a generating $\Gamma$-magic map exists if and only if $\BG=0$, which is equivalent to $r=1$ and hence to $\Gamma$ being cyclic.
\end{proof}

For elementary abelian Cayley groups, the same divisibility condition gives an exact distance-magic existence criterion.

\begin{theorem}\label{thm:elementary-cayley}
Let $p$ be a prime and $n\ge1$. Let $P\cong(\Z/p\Z)^n$, and let $G=\operatorname{Cay}(P,S)$ have regularity $k=|S|$. The following conditions are equivalent.
\begin{enumerate}[label=\emph{(\roman*)},leftmargin=2.3em]
\item The graph $G$ admits a $\Gamma$-distance magic labeling for some abelian group $\Gamma$ of order $|P|$.
\item The graph $G$ is $P$-distance magic.
\item The regularity $k$ is divisible by $p$.
\end{enumerate}
Thus, if $p\nmid k$, no abelian group of order $|P|$ gives a distance magic labeling, while if $p\mid k$, the prescribed group $P$ does.
\end{theorem}

\begin{proof}
Part~\emph{(ii)} implies part~\emph{(i)}. If part~\emph{(i)} holds and $p\nmid k$, Theorem~\ref{thm:cayley-p} gives a contradiction. Hence part~\emph{(i)} implies part~\emph{(iii)}.

Suppose that $p\mid k$, and regard $P$ additively. Define
\[
  f\colon P\longrightarrow P,
  \qquad
  f(x)=x.
\]
For every $x\in P$, the sum of the labels on its neighborhood is
\[
  \sum_{s\in S}f(x+s)
  =kx+\sum_{s\in S}s
  =\sum_{s\in S}s,
\]
because $p\mid k$ and $\ex(P)=p$. Thus $f$ is a bijective $P$-magic map, so $G$ is $P$-distance magic. This proves that part~\emph{(iii)} implies part~\emph{(ii)}. The same implication also follows from~\cite[Theorem~2.5(b)]{CFSZ16}.
\end{proof}

For $q\ge2$, let $\mathcal A$ be a set with $q$ elements. For $x,y\in\mathcal A^d$, their \emph{Hamming distance} is
\[
  d_H(x,y)=\bigl|\{i:x_i\ne y_i\}\bigr|.
\]
When $\mathcal A=\F_q$, the \emph{Hamming weight} of $x\in\F_q^d$ is
\[
  \operatorname{wt}_H(x)=\bigl|\{i:x_i\ne0\}\bigr|=d_H(x,0).
\]
For $1\le j\le d$, let $H_j(d,q)$ denote the $j$th relation graph of the Hamming scheme $H(d,q)$; that is, $H_j(d,q)$ has vertex set $\mathcal A^d$, with two vertices adjacent if and only if $d_H(x,y)=j$. In particular, $H_1(d,q)$ is the usual Hamming graph.

\begin{corollary}\label{cor:hamming-relations-exact}
Let $q=p^a$ be a prime power and $1\le j\le d$. The graph $H_j(d,q)$ admits a $\Gamma$-distance magic labeling for some abelian group $\Gamma$ of order $q^d$ if and only if
\[
  p\mid\binom dj.
\]

\noindent More precisely, if $p\mid\binom dj$, then $H_j(d,q)$ is $\F_q^d$-distance magic. If $p\nmid\binom dj$, then, for every abelian group $\Gamma$ of order $q^d$, the graph admits no affinely generating $\Gamma$-magic map, and it admits a generating $\Gamma$-magic map if and only if $\Gamma$ is cyclic.
\end{corollary}

\begin{proof}
Identify $\mathcal A$ with $\F_q$. Then $H_j(d,q)$ is a Cayley graph on the elementary abelian $p$-group $\F_q^d$, and its connection set consists of the vectors of Hamming weight $j$. Hence its regularity is
\[
  \binom dj(q-1)^j.
\]
Since $p\nmid q-1$, this regularity is divisible by $p$ exactly when $p\mid\binom dj$. If $p\mid\binom dj$, Theorem~\ref{thm:elementary-cayley} gives an $\F_q^d$-distance magic labeling. If $p\nmid\binom dj$, Theorem~\ref{thm:cayley-p} gives the remaining assertions.
\end{proof}

For a positive integer $n$, write
\[
  \rad(n)=\prod_{p\mid n}p
\]
for the product of its distinct prime divisors. The preceding corollary gives an exact criterion for all Hamming relations when $q$ is a prime power. For arbitrary $q$, the following necessary condition holds.

\begin{theorem}\label{thm:hamming-screen}
Let $q\ge2$ and $1\le j\le d$. If $H_j(d,q)$ admits a $\Gamma$-distance magic labeling for an abelian group $\Gamma$ of order $q^d$, then
\[
  \rad(q)\mid\binom dj.
\]
\end{theorem}

\begin{proof}
The adjacency eigenvalues of $H_j(d,q)$ are the $q$-ary Krawtchouk values
\[
  K_j(i)=\sum_{h=0}^j(-1)^h(q-1)^{j-h}
  \binom ih\binom{d-i}{j-h},
  \qquad 0\le i\le d,
\]
with multiplicities $\binom di(q-1)^i$~\cite{BCN89}. Let $p$ be a prime divisor of $q$. Since $q-1\equiv-1\pmod p$, Vandermonde's identity gives
\[
\begin{aligned}
  K_j(i)
  &\equiv (-1)^j\sum_{h=0}^j
  \binom ih\binom{d-i}{j-h}\\
  &=(-1)^j\binom dj
  \pmod p
\end{aligned}
\]
for every $0\le i\le d$.

Suppose that $p\nmid\binom dj$. Then every adjacency eigenvalue is nonzero modulo $p$, so $\det A\not\equiv0\pmod p$ and $A$ is invertible over $\F_p$. The graph is regular of degree
\[
  k=\binom dj(q-1)^j,
\]
and
\[
  k\equiv(-1)^j\binom dj\not\equiv0\pmod p.
\]
Hence $A$ acts invertibly on the constant subspace $\F_p\mathbf1$. Since this subspace is invariant, the induced operator $\overline A_p$ on the quotient is invertible as well.

Since $p\mid|\Gamma|$, the vector space $\operatorname{Hom}(\Gamma,\F_p)$ is nonzero. Lemma~\ref{lem:affine-param}\emph{(a)} with $m=p$ gives an injection of this nonzero space into $\ker\overline A_p$, a contradiction. Thus $p\mid\binom dj$ for every prime divisor $p$ of $q$, and therefore $\rad(q)\mid\binom dj$.
\end{proof}

For $d\ge1$, write $Q_d=H_1(d,2)$ for the $d$-dimensional hypercube. For $d\ge2$, the graph $H_2(d,2)$ has two connected components, induced respectively by the binary vectors of even and odd Hamming weight. These components are isomorphic; either one is called the \emph{halved $d$-cube} and is denoted by $\frac12Q_d$.

\begin{corollary}\label{cor:hamming-special-families}
The following consequences hold.
\begin{enumerate}[label=\emph{(\alph*)},leftmargin=2.2em]
\item Let $d\ge2$. The halved cube $\frac12Q_d$ admits a $\Gamma$-distance magic labeling for some abelian group $\Gamma$ of order $2^{d-1}$ if and only if $d\equiv0$ or $1\pmod4$.

\item For $q\ge2$, the rook graph
\[
  R_q=K_q\square K_q=H_1(2,q)
\]
admits a $\Gamma$-distance magic labeling for some abelian group $\Gamma$ of order $q^2$ if and only if $q$ is a power of $2$.

\item For $q\ge2$, the complement of the rook graph satisfies
\[
  \overline{R_q}=H_2(2,q)
\]
and admits no $\Gamma$-distance magic labeling for any abelian group $\Gamma$ of order $q^2$.

\item If $d\ge2$ is even, then the orthogonality graph
\[
  \Omega_d=H_{d/2}(d,2)
\]
admits a $\Gamma$-distance magic labeling for some abelian group $\Gamma$ of order $2^d$.
\end{enumerate}
\end{corollary}

The verification is immediate from Theorem~\ref{thm:elementary-cayley}, Corollary~\ref{cor:hamming-relations-exact}, and Theorem~\ref{thm:hamming-screen}, and is left to the reader.

\begin{remark}\label{rem:prescribed-gamma}
Corollaries~\ref{cor:hamming-relations-exact} and~\ref{cor:hamming-special-families} classify existence after the labeling group is allowed to vary over all abelian groups of the graph order. In each feasible case the elementary abelian Cayley group supplies a labeling. They do not give a sufficiency criterion for an arbitrary prescribed labeling group.

The hypercube is a family for which the prescribed-group problem is known completely. Anholcer, Cichacz, Froncek, Simanjuntak, and Qiu proved that, for every abelian group $\Gamma$ of order $2^d$,
\[
  Q_d\text{ is }\Gamma\text{-distance magic}
  \quad\Longleftrightarrow\quad
  d\text{ is even}
\]
\cite{ACFSQ21}. When $d$ is odd, the nonexistence direction also follows from Theorem~\ref{thm:cayley-p}, because $Q_d$ is the Cayley graph on $(\Z/2\Z)^d$ with connection set given by the standard basis and hence has regularity $d$.
\end{remark}

Theorem~\ref{thm:cayley-p} also gives restrictions for several further Cayley families. We first fix the notation and regularities needed for the standard forms graphs.

Let $q$ be a prime power. For positive integers $d$ and $e$, we write $\operatorname{Bil}(d,e;q)$ for the graph whose vertices are the $d\times e$ matrices over $\F_q$, with two matrices adjacent when their difference has rank $1$. For $n\ge2$, the graph $\operatorname{Alt}(n,q)$ has as vertices the alternating $n\times n$ matrices over $\F_q$, with adjacency when the difference has rank $2$. For $n\ge1$, the graph $\operatorname{Her}(n,q^2)$ has as vertices the Hermitian $n\times n$ matrices over $\F_{q^2}$, with adjacency when the difference has rank $1$.
For $n\ge1$, the graph $\operatorname{Sym}(n,q)$ has as vertices the symmetric $n\times n$ matrices over $\F_q$, with adjacency when the difference has rank $1$. Finally, for $n\ge1$, the graph $\operatorname{Qua}(n,q)$ has as vertices the quadratic forms in $n$ variables over $\F_q$, with two forms adjacent when the rank of their difference is $1$ or $2$. These graphs are Cayley graphs on the additive groups of their respective vertex spaces. The following table gives their regularities. See~\cite[Sections~9.5--9.6]{BCN89} for the definitions and parameter data.

\begin{center}
\footnotesize
\setlength{\tabcolsep}{4pt}
\renewcommand{\arraystretch}{1.22}
\begin{tabular}{@{}ccccc@{}}
\toprule
$\operatorname{Bil}(d,e;q)$ & $\operatorname{Alt}(n,q)$ & $\operatorname{Her}(n,q^2)$ & $\operatorname{Sym}(n,q)$ & $\operatorname{Qua}(n,q)$\\
\midrule
$\displaystyle\frac{(q^d-1)(q^e-1)}{q-1}$ &
$\displaystyle\frac{(q^n-1)(q^{n-1}-1)}{q^2-1}$ &
$\displaystyle\frac{q^{2n}-1}{q+1}$ &
$q^n-1$ &
$\displaystyle\frac{(q^{n+1}-1)(q^n-1)}{q^2-1}$\\
\bottomrule
\end{tabular}
\end{center}

\begin{corollary}\label{cor:cayley-families}
Let $q=p^a$ be a prime power. For each graph $G$ in the following list and every abelian group $\Gamma$ of order $|V(G)|$, the graph $G$ admits no affinely generating $\Gamma$-magic map and no $\Gamma$-distance magic labeling. Moreover, $G$ admits a generating $\Gamma$-magic map if and only if $\Gamma$ is cyclic.
\begin{enumerate}[label=\emph{(\alph*)},leftmargin=2.2em]
\item An undirected generalized Paley graph $\operatorname{Cay}(\F_q^+,H)$ with $H\le\F_q^\times$ and $-H=H$.
\item One of the forms graphs $\operatorname{Bil}(d,e;q)$, $\operatorname{Alt}(n,q)$, $\operatorname{Her}(n,q^2)$, $\operatorname{Sym}(n,q)$, or $\operatorname{Qua}(n,q)$ defined above.
\item The Doob graph $D(m,n)$, the Cartesian product of $m$ Shrikhande graphs and $n$ copies of $K_4$, where $m$ is a nonnegative integer and $n$ is a positive odd integer.
\end{enumerate}
\end{corollary}

\begin{proof}
In part~\emph{(a)}, the generalized Paley graph is a Cayley graph on $\F_q^+$ of regularity $|H|$, a divisor of $q-1$, hence prime to $p$.

For part~\emph{(b)}, the table above gives the regularity of each forms graph. Since $q\equiv0\pmod p$, while $q-1$, $q+1$, and $q^2-1$ are nonzero modulo $p$, every displayed regularity is congruent to $-1$ modulo $p$, and hence is prime to $p$.

Finally, $D(m,n)$ has an abelian $2$-group Cayley realization and regularity $6m+3n$~\cite{SHK19}, which is odd when $n$ is odd. In every case Theorem~\ref{thm:cayley-p} applies.
\end{proof}

\subsection{Strongly regular graphs}

A strongly regular graph with parameters $(v,k,\lambda,\mu)$ has $v$ vertices, is $k$-regular, every two adjacent vertices have $\lambda$ common neighbors, and every two nonadjacent vertices have $\mu$ common neighbors.
The case $k=\mu$ has a standard structural interpretation. A noncomplete strongly regular graph satisfies $k=\mu$ if and only if it is a complete multipartite graph with all parts of the same size. Write such a graph as $K_{m,\ldots,m}$, with $r$ parts, where $m,r\ge2$. The existence of group distance magic labelings for this family is completely determined. For an abelian group $\Gamma$ of order $rm$, the graph $K_{m,\ldots,m}$ is $\Gamma$-distance magic if and only if $m$ is even or $\Gamma$ does not have exactly one involution; consequently, it is group distance magic if and only if $m$ is even or $r$ is odd~\cite{Cic18}. We therefore restrict attention in this subsection to the case $k\ne\mu$.

We write $I$ and $J$ for the identity and all-one matrices of order $v$, respectively. The adjacency matrix of a strongly regular graph satisfies~\cite[Section~1.1.1]{BvM22}
\[
  A^2=(\lambda-\mu)A+(k-\mu)I+\mu J.
\]
Hence
\begin{equation}\label{eq:srg-relation}
  \overline A^2-(\lambda-\mu)\overline A-(k-\mu)I=0
\end{equation}
on $\Lamone$.

\begin{theorem}\label{thm:srg}
Let $G$ be strongly regular with parameters $(v,k,\lambda,\mu)$, where $k\ne\mu$, and let $\Gamma$ be a finite abelian group with invariant factors
\[
  d_1\mid\cdots\mid d_r.
\]
The following necessary conditions hold.
\begin{enumerate}[label=\emph{(\alph*)},leftmargin=2.2em]
\item If $G$ admits an affinely generating $\Gamma$-magic map, then
\[
  d_r=\ex(\Gamma)\mid k-\mu.
\]
\item If $r\ge2$ and $G$ admits a generating $\Gamma$-magic map, then
\[
  d_{r-1}\mid k-\mu.
\]
\end{enumerate}
In particular, if $|\Gamma|=v$ and $G$ admits a $\Gamma$-distance magic labeling, then
\[
  \ex(\Gamma)\mid k-\mu.
\]
\end{theorem}

\begin{proof}
Suppose first that $\overline A x=0$ in $\Lamone\otimes\mathbb Q$. Substituting into \eqref{eq:srg-relation} gives
\[
  (k-\mu)x=0.
\]
Since $k\ne\mu$, one has $x=0$. Hence $\overline A$ is nonsingular over $\mathbb Q$ and $\Sred(G)$ is finite.

Now pass relation~\eqref{eq:srg-relation} to the cokernel of $\overline A$. The terms containing $\overline A$ vanish there, so
\[
  (k-\mu)\Sred(G)=0.
\]
Thus the exponent of every subgroup of $\Sred(G)$ divides $k-\mu$. Corollary~\ref{cor:main-smith}\emph{(b)} embeds $\Gamma$ into $\Sred(G)$ under affine generation, giving part~\emph{(a)}. Corollary~\ref{cor:main-smith}\emph{(a)} embeds $\BG$ under ordinary generation, and $\ex(\BG)=d_{r-1}$ when $r\ge2$, giving part~\emph{(b)}.
The final assertion follows from part~\emph{(a)} because every distance magic labeling is affinely generating.
\end{proof}

As a consequence, if $p\mid v$ and $p\nmid(k-\mu)$, no abelian group $\Gamma$ of order $v$ admits an affinely generating $\Gamma$-magic map or a $\Gamma$-distance magic labeling.

We use the following standard terminology. A \emph{conference graph} with parameter $t$ is a strongly regular graph with parameters $(4t+1,2t,t-1,t)$. Let $V=\F_2^{2\nu}$, and fix a nondegenerate alternating bilinear form $\langle\cdot,\cdot\rangle$ on $V$. The \emph{symplectic graph} $\operatorname{Sp}(2\nu,2)$ has vertex set $V\setminus\{0\}$, with distinct vertices $x$ and $y$ adjacent when $\langle x,y\rangle=1$. It is strongly regular with parameters $(2^{2\nu}-1,2^{2\nu-1},2^{2\nu-2},2^{2\nu-2})$. A \emph{Moore graph of diameter two and degree $k$} is a $k$-regular graph of diameter $2$ and girth $5$. Equivalently, it is strongly regular with parameters $(k^2+1,k,0,1)$. See~\cite{BvM22} for these families.

The following three families admit a stronger simplification of Theorem~\ref{thm:srg}.

\begin{corollary}\label{cor:srg-cyclic}
For a conference graph with parameters $(4t+1,2t,t-1,t)$ and $t\ge1$, a symplectic graph $\operatorname{Sp}(2\nu,2)$ with $\nu\ge2$, or a Moore graph of diameter two and degree $k\ge2$, and for every abelian group $\Gamma$ of the graph order, there is no affinely generating $\Gamma$-magic map and hence no $\Gamma$-distance magic labeling, while a generating $\Gamma$-magic map exists if and only if $\Gamma$ is cyclic.
\end{corollary}

\begin{proof}
Let $v$ denote the order of the graph.
For a conference graph, $\gcd(4t+1,t)=1$. For a symplectic graph $\operatorname{Sp}(2\nu,2)$ with $\nu\ge2$, one has $v=2^{2\nu}-1$ while $k-\mu=2^{2\nu-2}$. In both cases no prime divides both $v$ and $k-\mu$. Since $\ex(\Gamma)$ and $|\Gamma|=v$ have the same prime divisors, Theorem~\ref{thm:srg}\emph{(a)} excludes affine generation and hence distance magic labelings. If $\Gamma$ were noncyclic, then any prime divisor of $d_{r-1}>1$ would divide both $v$ and $k-\mu$, contradicting Theorem~\ref{thm:srg}\emph{(b)}.

For a Moore graph, $k-\mu=k-1$ and $\gcd(k^2+1,k-1)$ divides $2$. If affine generation existed, every prime divisor of $v=k^2+1$ would divide this greatest common divisor, so $k^2+1$ would be a power of $2$. This is impossible for $k\ge2$, since $k^2+1$ is odd for even $k$ and congruent to $2$ modulo $4$ for odd $k$. If ordinary generation existed over a noncyclic group, then a prime $p\mid d_{r-1}$ would divide $k-1$, and $d_{r-1}\mid d_r$ would give $p^2\mid k^2+1$. The displayed greatest common divisor forces $p=2$ and hence $4\mid k^2+1$, again impossible.

In all cases cyclic groups admit the constant-generator map.
\end{proof}

We now introduce the further standard families that appear, together with those of Corollary~\ref{cor:srg-cyclic}, in Table~\ref{tab:srg}. The \emph{triangular graph} $T(n)$ has the $2$-subsets of an $n$-element set as its vertices, with two vertices adjacent when they intersect. Equivalently, $T(n)$ is the line graph of $K_n$, and it is strongly regular with parameters $(\binom n2,2(n-2),n-2,4)$. The rook graph $R_n=K_n\square K_n$ was defined above and is strongly regular with parameters $(n^2,2(n-1),n-2,2)$.

A \emph{Steiner triple system} $STS(u)$ consists of a set of $u$ points and a collection of $3$-subsets, called blocks, such that every pair of points lies in exactly one block. Its \emph{block graph} has the blocks as vertices, with two blocks adjacent when they intersect. It is strongly regular with parameters $\left(\frac{u(u-1)}6,\frac{3(u-3)}2,\frac{u+3}2,9\right)$.

A strongly regular graph with parameter set $(n^2,s(n-1),n+s^2-3s,s(s-1))$ is said to be of \emph{Latin square type}. One with parameter set $(n^2,s(n+1),-n+s^2+3s,s(s+1))$ is said to be of \emph{negative Latin square type}.

A \emph{partial geometry} $pg(s,t,\alpha)$ is a point-line incidence structure in which every line contains $s+1$ points and every point lies on $t+1$ lines. Any two points lie on at most one line. If a point does not lie on a line, exactly $\alpha$ lines through the point meet that line. The \emph{point graph} has the points as vertices, with two vertices adjacent when they are collinear, and is strongly regular with parameters $\left(\frac{(s+1)(st+\alpha)}{\alpha},\ s(t+1),\ s-1+t(\alpha-1),\ \alpha(t+1)\right)$.

A \emph{generalized quadrangle} $GQ(s,t)$ is a partial geometry $pg(s,t,1)$. The generalized-quadrangle row below refers to its point graph. These definitions and parameter conventions are standard~\cite{BvM22}.

Let $\Gamma$ be an abelian group of the graph order with invariant factors $d_1\mid\cdots\mid d_r$. Theorem~\ref{thm:srg} gives the family-specific nonexistence conditions in Table~\ref{tab:srg}. The first three rows record Corollary~\ref{cor:srg-cyclic}. For the remaining families, the entries follow by substituting the corresponding value of $k-\mu$ into Theorem~\ref{thm:srg}. The verification of these entries is left to the reader. Since every $\Gamma$-distance magic labeling is affinely generating, every condition in the last column also rules out a $\Gamma$-distance magic labeling. The values of $k-\mu$ are obtained from the displayed parameter sets and the standard parameter sets in~\cite{BvM22}.

\begin{table}[ht]
\caption{\centering Nonexistence conditions for standard strongly regular families, where $\Gamma$ is an abelian group of the graph order.}
\label{tab:srg}
\footnotesize
\setlength{\tabcolsep}{2pt}
\renewcommand{\arraystretch}{1.5}
\begin{tabular}{@{}
>{\raggedright\arraybackslash}p{0.234\textwidth}
@{\hspace{7pt}}
>{\raggedright\arraybackslash}p{0.14\textwidth}
@{\hspace{7pt}}
>{\raggedright\arraybackslash}p{0.258\textwidth}
@{\hspace{11pt}}
>{\raggedright\arraybackslash}p{0.27\textwidth}@{}}
\toprule
Family & $k-\mu$ & No generating $\Gamma$-magic map if & No affinely generating $\Gamma$-magic map if\\
\midrule
Conference graph, $t\ge1$
  & $t$
  & $\Gamma$ is noncyclic
  & always\\
Symplectic graph, $\nu\ge2$
  & $2^{2\nu-2}$
  & $\Gamma$ is noncyclic
  & always\\
Moore graph of diameter $2$, $k\ge2$
  & $k-1$
  & $\Gamma$ is noncyclic
  & always\\
Triangular graph $T(n)$, $n\ge5$
  & $2(n-4)$
  & $r\ge2$ and $d_{r-1}\nmid2(n-4)$
  & $d_r\nmid2(n-4)$\\
Rook graph $R_n=K_n\square K_n$, $n\ge3$
  & $2(n-2)$
  & $r\ge2$ and $d_{r-1}\nmid2(n-2)$
  & $d_r\nmid2(n-2)$\\
Block graph of an $STS(u)$, $u\ge13$
  & $\dfrac{3(u-9)}2$
  & $r\ge2$ and $d_{r-1}\nmid\dfrac{3(u-9)}2$
  & $d_r\nmid\dfrac{3(u-9)}2$\\
Latin square type, $s\ne n$
  & $s(n-s)$
  & $r\ge2$ and $d_{r-1}\nmid s(n-s)$
  & $d_r\nmid s(n-s)$\\
Negative Latin square type, $s\ne n$
  & $s(n-s)$
  & $r\ge2$ and $d_{r-1}\nmid s(n-s)$
  & $d_r\nmid s(n-s)$\\
Point graph of a $pg(s,t,\alpha)$, $s\ne\alpha$
  & $(s-\alpha)(t+1)$
  & $r\ge2$ and $d_{r-1}\nmid(s-\alpha)(t+1)$
  & $d_r\nmid(s-\alpha)(t+1)$\\
Generalized quadrangle $GQ(s,t)$, $s>1$
  & $(s-1)(t+1)$
  & $r\ge2$ and $d_{r-1}\nmid(s-1)(t+1)$
  & $d_r\nmid(s-1)(t+1)$\\
\bottomrule
\end{tabular}
\end{table}

\subsection{Incidence graphs of symmetric designs}

For background on symmetric designs and their incidence matrices, see~\cite[Section~8.3]{BvM22}. A symmetric $2$-$(v,k,\lambda)$ design has $v$ points and $v$ blocks, every block contains $k$ points, every point lies in $k$ blocks, and every pair of distinct points lies in exactly $\lambda$ blocks. Let $\mathcal D$ be a nontrivial symmetric $2$-$(v,k,\lambda)$ design with incidence matrix $B$. Its incidence graph is the bipartite graph on the point and block sets, with adjacency given by incidence, and has adjacency matrix
\[
  A=\begin{pmatrix}0&B\\B^{\mathsf T}&0\end{pmatrix}
\]
and
\[
  BB^{\mathsf T}=(k-\lambda)I+\lambda J.
\]

\begin{theorem}\label{thm:design}
Let $G$ be the incidence graph of a nontrivial symmetric $2$-$(v,k,\lambda)$ design, and let $\Gamma$ be a finite abelian group with invariant factors $d_1\mid\cdots\mid d_r$.
\begin{enumerate}[label=\emph{(\alph*)},leftmargin=2.2em]
\item If $G$ admits an affinely generating $\Gamma$-magic map, then
\[
  d_r=\ex(\Gamma)\mid k(k-\lambda).
\]
\item If $r\ge2$ and $G$ admits a generating $\Gamma$-magic map, then
\[
  d_{r-1}\mid k(k-\lambda).
\]
\end{enumerate}
In particular, if $|\Gamma|=2v$ and $G$ admits a $\Gamma$-distance magic labeling, then
\[
  \Gamma\hookrightarrow\Sred(G)
  \qquad\text{and}\qquad
  2v\mid k(k-\lambda)^{v-1}.
\]
\end{theorem}

\begin{proof}
Let $\mathbf1_P$ and $\mathbf1_B$ denote the all-one vectors on the point and block parts, so that
\[
  B\mathbf1_B=k\mathbf1_P
  \qquad\text{and}\qquad
  B^{\mathsf T}\mathbf1_P=k\mathbf1_B.
\]
The vector $(\mathbf1_P,\mathbf1_B)$ is the global constant vector and disappears in the reduced quotient. Write $u$ for the class of $(\mathbf1_P,0)$ in $\Lamone$. Then $(\mathbf1_P,0)\equiv-(0,\mathbf1_B)$ modulo the constants, and the second displayed identity gives
\[
  \overline A u=-ku.
\]

For a symmetric design one has $BB^{\mathsf T}=B^{\mathsf T}B=(k-\lambda)I+\lambda J$, hence
\[
  A^2=(k-\lambda)I+\lambda
  \begin{pmatrix}J&0\\0&J\end{pmatrix}.
\]
The operator on the right sends $(x,y)$ to $\bigl((\textstyle\sum x)\mathbf1_P,(\textstyle\sum y)\mathbf1_B\bigr)$, whose class in $\Lamone$ is an integral multiple of $u$. Therefore $\overline A^2-(k-\lambda)I$ has image contained in $\Z u$, and since $\overline A u=-ku$,
\[
  (\overline A+kI)\bigl(\overline A^2-(k-\lambda)I\bigr)=0
\]
as an identity of endomorphisms of the lattice $\Lamone$. The corresponding polynomial $(x+k)\bigl(x^2-(k-\lambda)\bigr)$ has constant term $-k(k-\lambda)$, so all its other terms vanish in $\operatorname{coker}\overline A$ and
\[
  k(k-\lambda)\Sred(G)=0.
\]
The reduced eigenvalues are $-k$ once and $\pm\sqrt{k-\lambda}$, each with multiplicity $v-1$. In particular, $\overline A$ is nonsingular over $\mathbb Q$, so $\Sred(G)$ is finite. Parts~\emph{(a)} and \emph{(b)} now follow from Corollary~\ref{cor:main-smith}\emph{(b)} and \emph{(a)}, respectively.

From the same eigenvalues,
\[
  |\Sred(G)|=|\det\overline A|=k(k-\lambda)^{v-1}.
\]
The final assertion now follows because a $\Gamma$-distance magic labeling is affinely generating. Corollary~\ref{cor:main-smith}\emph{(b)} gives the displayed embedding, and Lagrange's theorem gives the order divisibility.
\end{proof}

\begin{corollary}\label{cor:design-coprime}
Let $G$ be the incidence graph of a nontrivial symmetric $2$-$(v,k,\lambda)$ design. If $\gcd(v,k)=1$, then, for every abelian group $\Gamma$ of order $2v$, there is no affinely generating $\Gamma$-magic map and hence no $\Gamma$-distance magic labeling, while a generating $\Gamma$-magic map exists if and only if $\Gamma$ is cyclic.
\end{corollary}

\begin{proof}

Assume $\gcd(v,k)=1$. Let $p$ be a prime divisor of $v$. The design identity
\[
  k(k-1)=\lambda(v-1)
\]
gives
\[
  k-\lambda\equiv k^2\pmod p.
\]
Thus
\[
  p\mid k(k-\lambda)
  \quad\Longleftrightarrow\quad
  p\mid k.
\]
Since $p\nmid k$, the integer $k(k-\lambda)$ is not divisible by $p$. Every group of order $2v$ has exponent divisible by $p$, so Theorem~\ref{thm:design}\emph{(a)} excludes affine generation.

Suppose now that $\Gamma$ is noncyclic and that $G$ admits a generating $\Gamma$-magic map. Choose a prime $p\mid d_{r-1}$. Since $d_{r-1}\mid d_r$, one has $p^2\mid2v$, and hence $p\mid v$. Theorem~\ref{thm:design}\emph{(b)} gives $p\mid k(k-\lambda)$, contradicting the preceding calculation. Therefore no generating $\Gamma$-magic map exists when $\Gamma$ is noncyclic. The converse for cyclic groups follows from the constant-generator construction.
\end{proof}

\begin{corollary}\label{cor:design-parity}
Let $G$ be the incidence graph of a nontrivial symmetric $2$-$(v,k,\lambda)$ design. If $k$ is odd and $\lambda$ is even, then, for every finite abelian group $\Gamma$ of even order, there is no affinely generating $\Gamma$-magic map. Hence there is no $\Gamma$-distance magic labeling for any abelian group $\Gamma$ of order $2v$.
\end{corollary}

\begin{proof}
Both $k$ and $k-\lambda$ are odd, so $k(k-\lambda)$ is odd. If $|\Gamma|$ is even, then $\ex(\Gamma)$ is even and therefore cannot divide $k(k-\lambda)$. Theorem~\ref{thm:design}\emph{(a)} excludes affine generation, and the distance-magic assertion follows.
\end{proof}

We recall the design families used below. A finite projective plane of order $q$ is a symmetric $2$-$(q^2+q+1,q+1,1)$ design. For a prime power $q$ and $m\ge2$, let $PG(m,q)$ denote the $m$-dimensional projective space over $\F_q$. Its points and hyperplanes, with incidence given by containment, form a symmetric $2$-$\bigl(\sum_{i=0}^{m}q^i,\sum_{i=0}^{m-1}q^i,\sum_{i=0}^{m-2}q^i\bigr)$ design.
A \emph{Hadamard symmetric design} has parameters $2$-$(4u-1,2u-1,u-1)$. Its complementary design has parameters $2$-$(4u-1,2u,u)$. See~\cite{BvM22} for background on these families.

\begin{corollary}\label{cor:design-families}
For each of the following graphs and every abelian group $\Gamma$ of the same order, there is no affinely generating $\Gamma$-magic map and hence no $\Gamma$-distance magic labeling, while a generating $\Gamma$-magic map exists if and only if $\Gamma$ is cyclic.
\begin{enumerate}[label=\emph{(\alph*)},leftmargin=2.2em]
\item The incidence graph of a finite projective plane.
\item The incidence graph of the point-hyperplane design of $PG(m,q)$ with $m\ge2$.
\item The incidence graph of a Hadamard symmetric design $2$-$(4u-1,2u-1,u-1)$ with $u\ge2$.
\item The incidence graph of the complementary Hadamard symmetric design $2$-$(4u-1,2u,u)$ with $u\ge2$.
\end{enumerate}
\end{corollary}

\begin{proof}
For a projective plane, $v=q^2+q+1$ and $k=q+1$, so $v-qk=1$. For the point-hyperplane design of $PG(m,q)$,
\[
  v=1+q+\cdots+q^m
  \qquad\text{and}\qquad
  k=1+q+\cdots+q^{m-1},
\]
so again $v-qk=1$. For the two Hadamard parameter sets, one has
\[
  \gcd(4u-1,2u-1)=1
  \qquad\text{and}\qquad
  \gcd(4u-1,2u)=1.
\]
Thus $\gcd(v,k)=1$ in every case, and Corollary~\ref{cor:design-coprime} gives the stated conclusions in every case.
\end{proof}

%======================================================================
\appendix
\section{The reduced adjacency matrix for the counterexample}\label{app:matrix}

For completeness, we record the matrix used in the proof of Theorem~\ref{thm:affine-counterexample}. With the vertices ordered as $1,\ldots,27$ and the basis $[e_1],\ldots,[e_{26}]$ of $\Lamone$, the reduced adjacency operator is represented by
\[
  M=(a_{ij}-a_{27,j})_{1\le i,j\le26}.
\]
Explicitly, with rows and columns indexed in the natural order $1,\ldots,26$,
\begingroup
\tiny
\setlength{\arraycolsep}{1.55pt}
\[
M=\left(\begin{array}{*{26}{r}}
  0 & 0 & 0 & 0 & -1 & -1 & 1 & 0 & 1 & 0 & 1 & 0 & -1 & 1 & -1 & 0 & 0 & 1 & -1 & 0 & 0 & 0 & 0 & 0 & 1 & -1 \\
  0 & 0 & 0 & 0 & -1 & -1 & 0 & 1 & 0 & 0 & 1 & 1 & 0 & 0 & -1 & 0 & 1 & 0 & -1 & 0 & 0 & 0 & 1 & 0 & 0 & -1 \\
  0 & 0 & 0 & 0 & 0 & 0 & 0 & 0 & 0 & 1 & 0 & 0 & -1 & 0 & -1 & 0 & 0 & 0 & 0 & 1 & 0 & 0 & 0 & 1 & 0 & -1 \\
  0 & 0 & 0 & 0 & -1 & -1 & 1 & 1 & 1 & 0 & 0 & 0 & -1 & 0 & 0 & 0 & 0 & 0 & 0 & 0 & 0 & 0 & 0 & 0 & 0 & 0 \\
  0 & 0 & 1 & 0 & -1 & 0 & 0 & 0 & 0 & 1 & 0 & 0 & -1 & 0 & -1 & 1 & 0 & 0 & -1 & 0 & 0 & 1 & 0 & 0 & 0 & -1 \\
  0 & 0 & 1 & 0 & 0 & -1 & 0 & 0 & 0 & 0 & 0 & 1 & -1 & 1 & -1 & 0 & 0 & 0 & -1 & 0 & 0 & 0 & 1 & 0 & 0 & -1 \\
  1 & 0 & 0 & 1 & -1 & -1 & 0 & 0 & 0 & 0 & 0 & 0 & 0 & 0 & -1 & 0 & 0 & 0 & 0 & 0 & 1 & 0 & 0 & 0 & 0 & 0 \\
  0 & 1 & 0 & 1 & -1 & -1 & 0 & 0 & 0 & 0 & 0 & 0 & -1 & 0 & 0 & 0 & 0 & 1 & -1 & 0 & 1 & 0 & 0 & 1 & 0 & -1 \\
  1 & 0 & 0 & 1 & -1 & -1 & 0 & 0 & 0 & 0 & 0 & 0 & -1 & 0 & -1 & 1 & 0 & 0 & -1 & 1 & 1 & 1 & 0 & 0 & 0 & -1 \\
  0 & 0 & 1 & 0 & 0 & -1 & 0 & 0 & 0 & 0 & 0 & 1 & 0 & 0 & -1 & 0 & 0 & 0 & -1 & 0 & 0 & 0 & 0 & 0 & 1 & 0 \\
  1 & 1 & 0 & 0 & -1 & -1 & 0 & 0 & 0 & 0 & 0 & 0 & -1 & 0 & -1 & 1 & 1 & 0 & -1 & 0 & 0 & 1 & 0 & 0 & 0 & 0 \\
  0 & 1 & 0 & 0 & -1 & 0 & 0 & 0 & 0 & 1 & 0 & 0 & -1 & 0 & -1 & 0 & 1 & 0 & -1 & 0 & 0 & 0 & 0 & 1 & 1 & -1 \\
  0 & 1 & 0 & 0 & -1 & -1 & 1 & 0 & 0 & 1 & 0 & 0 & -1 & 0 & 0 & 0 & 0 & 0 & -1 & 0 & 0 & 0 & 1 & 0 & 0 & -1 \\
  1 & 0 & 0 & 0 & -1 & 0 & 0 & 0 & 0 & 0 & 0 & 0 & -1 & 0 & 0 & 0 & 1 & 0 & -1 & 1 & 0 & 0 & 0 & 0 & 1 & -1 \\
  0 & 0 & 0 & 1 & -1 & -1 & 0 & 1 & 0 & 0 & 0 & 0 & 0 & 1 & -1 & 0 & 0 & 1 & -1 & 0 & 0 & 0 & 0 & 0 & 0 & -1 \\
  0 & 0 & 0 & 0 & 0 & -1 & 0 & 0 & 1 & 0 & 1 & 0 & -1 & 0 & -1 & 0 & 0 & 1 & -1 & 1 & 1 & 0 & 0 & 0 & 0 & -1 \\
  0 & 1 & 0 & 0 & -1 & -1 & 0 & 0 & 0 & 0 & 1 & 1 & -1 & 1 & -1 & 0 & 0 & 0 & -1 & 0 & 0 & 1 & 1 & 0 & 0 & -1 \\
  1 & 0 & 0 & 0 & -1 & -1 & 0 & 1 & 0 & 0 & 0 & 0 & -1 & 0 & 0 & 1 & 0 & 0 & -1 & 1 & 0 & 0 & 0 & 1 & 0 & -1 \\
  0 & 0 & 1 & 1 & -1 & -1 & 1 & 0 & 0 & 0 & 0 & 0 & -1 & 0 & -1 & 0 & 0 & 0 & -1 & 0 & 1 & 1 & 0 & 0 & 0 & -1 \\
  0 & 0 & 1 & 0 & -1 & -1 & 0 & 0 & 1 & 0 & 0 & 0 & -1 & 1 & -1 & 1 & 0 & 1 & -1 & 0 & 0 & 0 & 0 & 1 & 0 & -1 \\
  0 & 0 & 0 & 0 & -1 & -1 & 1 & 1 & 1 & 0 & 0 & 0 & -1 & 0 & -1 & 1 & 0 & 0 & 0 & 0 & 0 & 0 & 0 & 0 & 1 & -1 \\
  0 & 0 & 0 & 0 & 0 & -1 & 0 & 0 & 1 & 0 & 1 & 0 & -1 & 0 & -1 & 0 & 1 & 0 & 0 & 0 & 0 & 0 & 1 & 0 & 0 & -1 \\
  0 & 1 & 0 & 0 & -1 & 0 & 0 & 0 & 0 & 0 & 0 & 0 & 0 & 0 & -1 & 0 & 1 & 0 & -1 & 0 & 0 & 1 & 0 & 1 & 0 & -1 \\
  0 & 0 & 1 & 0 & -1 & -1 & 0 & 1 & 0 & 0 & 0 & 1 & -1 & 0 & -1 & 0 & 0 & 1 & -1 & 1 & 0 & 0 & 1 & 0 & 0 & -1 \\
  1 & 0 & 0 & 0 & -1 & -1 & 0 & 0 & 0 & 1 & 0 & 1 & -1 & 1 & -1 & 0 & 0 & 0 & -1 & 0 & 1 & 0 & 0 & 0 & 0 & 0 \\
  0 & 0 & 0 & 1 & -1 & -1 & 1 & 0 & 0 & 1 & 1 & 0 & -1 & 0 & -1 & 0 & 0 & 0 & -1 & 0 & 0 & 0 & 0 & 0 & 1 & -1 \\
\end{array}\right).
\]
\endgroup
The minor $M'$ used in the proof is obtained from this matrix by retaining rows $3,\ldots,26$ and columns $1,\ldots,24$.

%======================================================================
\section{Computational verification}\label{app:verification}

The following Python code verifies the computational statements used in Theorem~\ref{thm:affine-counterexample} and Remark~\ref{rem:affine-counterexample-computation}. It reconstructs the graph from the neighborhood data in the theorem, verifies that the graph is simple and $6$-regular, checks that the identity labeling has magic constant $84$, forms the reduced adjacency matrix, verifies the two independent null vectors used in the proof, computes its ranks over $\F_3$ and $\mathbb Q$, and computes the determinant of the $24\times24$ minor used in the proof. The computation uses Python~3 and SymPy with exact arithmetic. No Smith normal form computation is required.

\begin{lstlisting}[language=Python]
from sympy import Matrix, GF
from sympy.polys.matrices import DomainMatrix

rows = [
 [7,9,11,14,18,25], [8,11,12,13,17,23], [5,6,10,19,20,24],
 [7,8,9,15,19,26], [3,6,10,16,22,27], [3,5,12,14,23,27],
 [1,4,13,19,21,26], [2,4,15,18,21,24], [1,4,16,20,21,22],
 [3,5,12,13,25,26], [1,2,16,17,22,26], [2,6,10,17,24,25],
 [2,7,10,15,23,27], [1,6,15,17,20,25], [4,8,13,14,18,27],
 [5,9,11,18,20,21], [2,11,12,14,22,23], [1,8,15,16,20,24],
 [3,4,7,21,22,27], [3,9,14,16,18,24], [7,8,9,16,19,25],
 [5,9,11,17,19,23], [2,6,13,17,22,24], [3,8,12,18,20,23],
 [1,10,12,14,21,26], [4,7,10,11,25,27], [5,6,13,15,19,26]
]
N = {i + 1: set(row) for i, row in enumerate(rows)}
V = range(1, 28)
assert all(len(N[v]) == 6 and v not in N[v] for v in V)
assert all(v in N[u] for v in V for u in N[v])

A = Matrix(27, 27, lambda i, j: int(j + 1 in N[i + 1]))
assert A == A.T and all(A[i, i] == 0 for i in range(27))
assert all(sum(A[i, j] for j in range(27)) == 6 for i in range(27))
assert all(x == 84 for x in A * Matrix(list(V)))

M = Matrix(26, 26, lambda i, j: A[i, j] - A[26, j])
u = Matrix([1,2,0,1,2,0,1,2,0,1,2,0,1,2,0,1,2,0,1,2,0,1,2,0,1,2])
v = Matrix([1,1,2,0,2,0,2,1,1,2,0,0,2,0,2,0,1,1,0,0,0,0,0,0,0,1])
assert all(int(x) % 3 == 0 for x in M * u)
assert all(int(x) % 3 == 0 for x in M * v)
assert not all(int(x) % 3 == 0 for x in u)
assert not all(int(x) % 3 == 0 for x in v)
assert (u[24] % 3, u[25] % 3) == (1, 2)
assert (v[24] % 3, v[25] % 3) == (0, 1)

rank_mod_3 = DomainMatrix.from_Matrix(M).convert_to(GF(3)).rank()
assert rank_mod_3 == 24
rank_over_Q = M.rank()
assert rank_over_Q == 25

Mp = M.extract(range(2, 26), range(24))
assert Mp.det() == -11
assert Mp.det() % 3 != 0
print("rank_F3(M) =", rank_mod_3)
print("rank_Q(M) =", rank_over_Q)
print("det(M') =", Mp.det())
\end{lstlisting}

%======================================================================

\end{document}